%% file: main.tex
\documentclass[11pt]{article}
\usepackage[T1]{fontenc}
\usepackage[margin=1in]{geometry}
\usepackage{multicol}
\usepackage{amssymb}
\usepackage{amsmath}
\usepackage{accents}
\usepackage{mathtools}
\usepackage{mathrsfs}
\usepackage{bm}
\usepackage{complexity}
\usepackage{xcolor}
\usepackage{graphicx}
\usepackage[caption=false]{subfig}
\usepackage{floatrow}
\newfloatcommand{capbtabbox}{table}[][\FBwidth]
\usepackage{soul}
\usepackage[ruled, linesnumbered, noend]{algorithm2e}
\usepackage{csquotes}
\usepackage{tikz}
\usetikzlibrary{positioning}
\usetikzlibrary{calc}
\usetikzlibrary{patterns}
\usetikzlibrary{arrows.meta}
\usetikzlibrary{3d}
\usetikzlibrary{decorations.pathmorphing}
\pgfdeclarelayer{bg}    
\pgfsetlayers{bg,main}  

\usepackage{amsthm}
\usepackage{thmtools,thm-restate}
\newtheorem{mresult}{Main Result}
\newtheorem{theorem}{Theorem}[section]
\newtheorem*{claim}{Claim}
\newtheorem{corollary}[theorem]{Corollary}
\newtheorem{lemma}[theorem]{Lemma}
\newtheorem{proposition}[theorem]{Proposition}
\newtheorem{definition}[theorem]{Definition}

\newtheorem{remark}[theorem]{Remark}

\newtheorem{example}[theorem]{Example}

\usepackage[
backend=biber,
style=alphabetic,
giveninits=true
]{biblatex}
\DeclareNameAlias{default}{family-given}
\usepackage{subfiles} 
\input{macros.tex}

\title{The Polymatroid Representation of a Greedoid, and Associated Galois Connections}
\author{Robert P. Streit\thanks{Department of Electrical and Computer Engineering at The University of Texas at Austin.} \and Vijay K. Garg\footnotemark[1]\textsuperscript{\;\:,}\thanks{Partially supported by NSF CNS-1812349, and the Cullen Trust for
Higher Education Endowed Professorship.}}
\date{\today}

\begin{document}

\maketitle

\begin{abstract}
    A \emph{greedoid} is a generalization of a \emph{matroid} allowing for more flexible analyses and modeling of combinatorial optimization problems.
    However, these structures decimate many matroid properties contributing to their pervasive nature.
    A \emph{polymatroid greedoid} \cite{korte1985polymatroid} presents an interesting middle ground,
    so we further develop this class.
    First we prove every \emph{local poset greedoid} for which the greedy algorithm correctly solves linear optimizations over its basic words must have a polymatroid representation.
    For this, we use relationships between the lattices of greedoid flats and closed sets of a polymatroid to generalize concepts in \cite{korte1985polymatroid}.
    Then, we show our generalization is defined by a Galois connection between the greedoid flats and closed sets of a representation.
    Finally, we apply this duality to identify a subclass of polymatroid greedoids with favorable properties, which we call \emph{strong polymatroid greedoids}.
    As technical tools for our analyses, we introduce \emph{optimism} and the \emph{Forking Lemma} for interval greedoids.
    Both are pervasive in our work, and are of independent interest.
\end{abstract}

\section{Introduction}

Simple algorithms are good, correct algorithms are better, and the harmony intersecting these two is best.
As a result, the greedy algorithm has inspired many axiomatic approaches towards studying when it ``works.'' 
One example is \emph{greedoids},  which are language generalizations of \emph{matroids}.
Specifically, the Edmonds-Rado Theorem \cite{edmonds1971matroids,rado1957note} says the greedy algorithm correctly optimizes any choice of linear function over a downward closed set family if and only if the family is a matroid.
Yet, the greedy algorithm is successful in many settings not captured by this statement.
Postulating the essence of this result lies in a matroid's \emph{exchange} property, Korte and Lovász propose greedoids in \cite{korte1984structural}.
But, such generalization gives up properties elevating matroids to their mathematical ubiquity, so greedoids do not enjoy as prominent a position in combinatorial optimization as matroids.
Still, greedoids show great potential.
Recent works give new applications to network reliability \cite{szeszler2021new} and generalized linear programs \cite{kempner2024greedoids}.
And more abstractly, there is opportunity to use greedoids to model complex structures in discrete programs, possibly to generalize matroids' role in submodular and discrete convex analysis.
But this requires more basic insights.
\emph{Polymatroid greedoids}, whose feasibility structure is encoded by polymatroids we call \emph{representations}, present a middle ground between greedoids and matroids.
So, we develop this class.

First, we generalize the approach from \cite{korte1985polymatroid} used to initially define polymatroid greedoids.
Though Korte and Lovász restrict their attention to integral polymatroid rank functions, we require more flexibility.
So, we introduce \emph{aligned representations} satisfying axioms we call \emph{alignment}, which asserts relationships between a greedoid's flats and a polymatroid's closed sets.
This meaningfully generalizes \cite{korte1985polymatroid} as 
Proposition~\ref{prop:int->aligned} proves integral representations are aligned.
Then, we examine linear optimization over greedoids, and show \emph{local poset greedoids} where the greedy algorithm is correct always have aligned representations.


\begin{mresult}\label{mr:lp-greedy=>pm-greedoid}
    Let $\glang$ be a greedoid with the local poset property,
    and suppose the greedy algorithm correctly solves any linear optimization over the basic words.
    Then, there exists an aligned polymatroid representation $\rho:2^\univ\to\mathbb{R}$ of $\glang$.
\end{mresult}

This advances our understanding of submodularity as a necessary condition for the greedy algorithm's correctness, as Theorem~\ref{thm:lp-greedy=>pm-greedoid} shows such local poset greedoids have structure conforming with a polymatroid.
Furthermore, our work and \cite{korte1985polymatroid} are evidence to representations being a powerful analytic tool, and by Theorem~\ref{thm:lp-greedy=>pm-greedoid} one need only verify the correctness of the greedy algorithm to apply them in analyzing of local poset greedoids.
Such is significantly easier a task than directly constructing an aligned representation, hence easing application to future work investigating discrete programs on greedoids with non-linear objectives.

Alignment is a minimal set of axioms ensuring that a polymatroid representation is ``useful.''
Our specificity is needed as a general continuous representation $\rho$ does fail to satisfy some properties, such as \emph{optimism} and the $\rho$-closure of the flat supports required for our second main result. 
Our second main result shows alignment in a representation $\rho$ is equivalent to the existence a special Galois connection between the lattices of flats $\L(\glang)$ and $\rho$-closed sets $\L(\rho)$, with adjoints determined by the kernel closure operator $\kappa$.
The result is technically surprising; for example, it is not obvious the image of $\kappa$ should even be in $\L(\rho)$.
Furthermore, this shows the flat support structure entirely determines aligned representations, exposing a new perspective on polymatroid greedoids and their representations.

\begin{mresult}\label{mr:gc}
    Let $\glang$ be a greedoid possessing optimism and the interval property, and $\rho:2^\univ\to\mathbb{R}$ a polymatroid representation.
    The representation $\rho$ is aligned if and only if the following are true:
    \begin{enumerate}
        \item The flat support sets are $\rho$-closed,
        \item The kernel closure operator $\kappa$ is cover-preserving,
        \item And, $\kappa:\L(\glang)\leftrightarrows\L(\rho):\kappa^{-1}$ is a Galois connection.
    \end{enumerate}
\end{mresult}

Finally, we identify a maximum aligned representation $\accentset{\vee}{\rho}$, called the \emph{greatest representation} and describe the subclass of polymatroid greedoids possessing it.
We call this subclass \emph{strong polymatroid greedoids}, and show optimistic interval greedoids with flat supports closed under intersection are exactly this class.
Our method uses the closure and interior operators of the Galois connection to show the $\accentset{\vee}{\rho}$-closed sets are isomorphic to the flats of any greedoids represented by $\accentset{\vee}{\rho}$.
Then, as the flat supports of an interval greedoid are also isomorphic to its flats, one can intuit that lattice of flat supports and $\accentset{\vee}{\rho}$-sets are actually equal (and so the meet of two flat supports is their intersection).
Similar techniques also show these greedoids have flats related via lattice-embedding to \emph{every} aligned representation's closed sets.
So strong polymatroid greedoids are quite well-behaved, since this makes the flats always isomorphic to a sublattice of the closed sets of any aligned representation.

\begin{mresult}\label{mr:strong-pm-greedoid}
    Let $\glang$ be an interval greedoid with optimism.
    The following are equivalent:
    \begin{enumerate}
        \item $\glang$ is a strong polymatroid greedoid,
        \item $\glang$ is a polymatroid greedoid such that for all aligned representations $\rho:2^\univ\to\mathbb{R}$, $\sigma_\rho\circ\kappa$ is a lattice-embedding of $\L(\glang)$ into $\L(\rho)$,
        \item The greatest representation $\accentset{\vee}{\rho}$ is a polymatroid representation of $\glang$,
        \item $\L(\glang)\cong\L(\accentset{\vee}{\rho})$, i.e. the $\accentset{\vee}{\rho}$-closed sets are isomorphic to the flats of $\glang$,
        \item The flat supports of $\glang$ form a Moore family.
    \end{enumerate}
\end{mresult}

We'll show that optimistic \emph{distributive supermatroids} (see \cite{tardos1990intersection}) are a nontrivial example of strong polymatroid greedoids.
See Figure~\ref{fig:tax}, where we display inclusions and intersections of major greedoid classes interacting with polymatroid greedoids and \emph{optimistic greedoids} (introduced in this work).
Any strict inclusions (or lack thereof) are verified in our appendices.

\begin{figure}[t!]
    \centering
    \scalebox{.65}{\input{tikz/optimistic.tikz}}
    \caption{
        By Theorem~\ref{thm:lp-greedy=>pm-greedoid}, the intersection of strong-exchange greedoids (those for which the greedy algorithm correctly solves linear optimizations \cite{goetschel1986linear,korte1984greedoids}) and local poset greedoids are greedoids with aligned representations.
    This intersection includes matroids and antimatroids with the local poset property.
    These examples are included in the class of strong polymatroid greedoids as their flat supports are closed under intersection (see Section~\ref{subsec:lattice-embeddings}).
    By Theorem~\ref{thm:pm->optimism}, polymatroid greedoids with aligned representations are contained in the intersection of optimistic and local poset greedoids.
    However, one can adapt an example from \cite{korte1988intersection} to show this inclusion is strict (see Appendix~\ref{app:local-augmentation}).
    Moreover, Appendix~\ref{app:optimism} shows all strong-exchange greedoids are optimistic, and gives a trimmed matroid construction from \cite{korte1988intersection} for the existence of an optimistic interval greedoid without the local poset property.
    Finally, Section~\ref{subsec:lattice-embeddings} shows optimistic distributive supermatroids are strong polymatroid greedoids, and all other shown inclusions pre-date our work (see \cite{korte2012greedoids}).}
    \label{fig:tax}
\end{figure}

\subsection{Organization}
We first review the prior art and preliminaries in Sections~\ref{sec:related-work} and~\ref{sec:preliminaries} (respectively).
Section~\ref{sec:pm-greedoid} broadly examines polymatroid greedoids, and proposes new mathematical tools needed for our analysis.
Section~\ref{subsec:alignment} defines aligned representations, while 
Section~\ref{subsec:optimistic} introduces optimism, a useful weakening of the monotonicity in a matroid's span. Theorem~\ref{thm:pm->optimism} proves greedoids with aligned representations are optimistic.
Section~\ref{subsec:forking} then presents the Forking Lemma, stating for all pairs $F$ and $F'$ of flats continuations of the meet leading to a flat strictly below $F'$ are continuations of $F$ (and vice versa).
This result provides witness of whether \emph{specific} letters can be given from one feasible word to another, in contrast to \emph{exchange} which only guarantees existence.
An immediate application is Proposition~\ref{prop:strong-opt}, which shows that optimistic interval greedoids satisfy a stronger form of optimism.

In Section~\ref{sec:characterization} we prove Main Result~\ref{mr:lp-greedy=>pm-greedoid}.
Specifically, we construct an aligned representation using a constrained optimization over the flats.
The objective function is defined with the basis rank, which is always a matroid rank function in this context \cite{korte1984greedoids}.
This is key to verifying this construction gives a polymatroid.
Finally, in Section~\ref{sec:galois} we begin by proving Main Result~\ref{mr:gc}.
From there, we introduce the greatest representation and strong polymatroid greedoids in Section~\ref{subsec:lattice-embeddings}, and prove Main Result~\ref{mr:strong-pm-greedoid} using the properties of the Galois connection.

\section{Related Work}\label{sec:related-work}

After observing the success of the greedy algorithm in a variety of settings incompatible with downward closed set families, Korte and Lovász introduced greedoids as an analytic tool generalizing matroids \cite{korte1984structural,korte2012greedoids}.
Succeeding research developed a now well understood taxonomy of various classes and their properties \cite{korte1986non,korte1985note,schmidt1991greedoids,schmidt1988characterization,boyd1990algorithmic,korte1988intersection,goecke1988greedy,brylawski1988exchange}.
For example, greedoids include \emph{antimatroids} encoding the constraint structure of scheduling problems \cite{boyd1990algorithmic}, \emph{undirected branching greedoids} encoding Prim's invariant partial solutions are connected sets covering a root when computing minimum spanning trees \cite{korte1984structural,schmidt1988characterization}, and even \emph{blossom greedoids} describing Edmond's matching algorithm \cite{korte1984structural,korte1986non}.
See \cite{korte2012greedoids} for an introduction and review.
Beyond this, greedoids have a rich mathematical structure that has inspired novel results in abstract convexity \cite{ahrens1999convexity,edelman1985theory}, lattice theory \cite{faigle1980geometries,crapo1984selectors,czedli2023revisiting}, discrete processes \cite{bjorner1991chip,bjorner2008random}, and topology \cite{bjorner1985homotopy}.
However, the primary motivation of greedoids was initially rooted in mathematical optimization.
This is best seen in \cite{korte1984greedoids,goetschel1986linear}, where a necessary and sufficient condition called \emph{strong exchange} is identified as characterizing those greedoids for which the greedy algorithm correctly solves a linear optimization.
Strong exchange is a prototypical manifestation of a \emph{matroid embedding}, later shown to completely characterize set families in which the greedy algorithm correctly optimizes linear objective functions \cite{helman1993exact}.
In absence of strong exchange, there does exist a class of \emph{admissible} objective functions for which the greedy algorithm is always correct in any greedoid \cite{korte2012greedoids,szeszler2022sufficient},
and Faigle examines the greedy algorithm
in ordered geometries \cite{faigle1980geometries}.
More recently, Szeszl{\'e}r investigated the polyhedral aspects and applications of \emph{local forest greedoids} (a subclass of local poset greedoids) \cite{szeszler2021new} and proved new insights about the greedy algorithm \cite{szeszler2022sufficient}.
Lastly, \cite{kempner2024greedoids} made connections between greedoids and \emph{violator spaces} \cite{amenta1993helly,gartner2008violator} used in the study of generalized linear programming.

Polymatroid greedoids \cite{korte1985polymatroid} are the starting point for our work, and we are the first to investigate any relationships between this class and the greedy algorithm.
Prior work examining the greedy algorithm in greedoids is very generalized, whereas we specialize Theorem~\ref{thm:pm->optimism} to local poset greedoids.
This leads to a result we feel is easier to apply, and creates new insights within the greedoid taxonomy (see Figure~\ref{fig:tax}).
Furthermore, our work enhances lattice approaches to greedoid theory, which has been mostly limited to interval greedoids \cite{crapo1984selectors,korte2012greedoids,faigle1980geometries} and antimatroids \cite{edelman1985theory} so far.
And our work shows new applications of the lattice theory of polymatroids, an understudied area recently examined in \cite{gustafson2023polymatroids}.
Finally, similarly motivated to study more ``matroid-like'' greedoids, \cite{brylawski1988exchange} introduced \emph{exchange systems}.
However, their interest lied in identifying greedoids maintaining matroid transformations (like erection, series-connection, lifts, duality, etc.), while our interest in polymatroid greedoids lies in their description via the structure of a submodular function.
Submodularity is an unavoidable topic in the discrete analysis associated with combinatorial optimization and integer programming (see \cite{fujishige2005submodular,murota1998discrete} for example), and so we feel this direction shows promise to future applications in optimization.

\section{Preliminaries}\label{sec:preliminaries}

Our primary sources for greedoids are \cite{korte2012greedoids} and \cite{bjorner1992introduction}.
We assume familiarity with matroids and partially ordered sets.
References for the former and latter are \cite{white1986theory,welsh2010matroid,oxley2006matroid} and \cite{caspard2012finite,davey2002introduction}.
Treatments of lattice theory can be found in \cite{davey2002introduction,gratzer2011lattice}, and applications to computer science and algebraic combinatorics can be found in \cite{garg2015introduction} and \cite{stanley2011enumerative}, respectively.
Now, for any set $X$ we let $X!$ give its permutations.
And, let $X + y$ be the union of $X$ with a singleton $\{y\}$, i.e.\ $X+y\triangleq X \cup \{y\}$, while the difference is $X - y \triangleq X\setminus\{y\}$.

\subsection{Lattices}
A partially ordered set (poset) $\mathscr{P} = (\mathcal{X}, \sqsubset)$ is defined by an irreflexive, asymmetric, and transitive relation over a ground set $\mathcal{X}$.
In this case we say that $\mathbf{x}$ is \emph{less than} or \emph{lies below} $\mathbf{y}$ whenever $\mathbf{x} \sqsubset \mathbf{y}$, and conversely that $\mathbf{x}$ and $\mathbf{y}$ are \emph{incomparable} if all of $\mathbf{x} \not\sqsubset \mathbf{y}$, $\mathbf{y} \not\sqsubset \mathbf{x}$, and $\mathbf{x} \neq \mathbf{y}$ hold true.
We denote the reflexive closure by $\sqsubseteq$, and say that $X \subseteq \mathscr{P}$ is an (order) \emph{ideal} if $\mathbf{y} \in X$ and $\mathbf{x} \sqsubseteq \mathbf{y}$ implies $\mathbf{x} \in X$.
An (order) \emph{filter} is defined dually.
For $\mathbf{x},\mathbf{y}\in\mathscr{P}$ we refer to the \emph{meet} (greatest lower bound) by $\mathbf{x} \sqcap \mathbf{y}$ and the \emph{join} (least upper bound) by $\mathbf{x} \sqcup \mathbf{y}$, when they exist.
A poset is a \emph{lattice} whenever it is closed under meet and join, making $\L=(\mathcal{X},\sqcap,\sqcup)$ a well defined algebraic structure.
Correspondingly, a \emph{sublattice} is a substructure closed under meet and join.
The \emph{covering relation}, which we denote by `$\prec$' for all partial orders in this work, is such that $\mathbf{x} \prec \mathbf{y}$ if and only if $\mathbf{x} \sqsubset \mathbf{y}$ and there is no $\mathbf{z}$ with $\mathbf{x} \sqsubset \mathbf{z} \sqsubset \mathbf{y}$.
The covering relation allows one to visualize posets with \emph{Hasse diagrams}, which consist of nodes corresponding to the elements of the ground set $\mathcal{X}$ and line segments such that there is a line directed upwards from $\mathbf{x}$ to $\mathbf{y}$ if and only if $\mathbf{x} \prec \mathbf{y}$.
A poset $\mathscr{P}$ is \emph{graded} whenever there exists a function $g:\mathscr{P}\to\mathbb{Z}_+$ consistent with the covering relation in the sense that $\mathbf{x} \prec \mathbf{y}$ implies $g(\mathbf{x}) = g(\mathbf{y}) - 1$.
A lattice is \emph{semimodular} if and only if $\mathbf{x} \succ \mathbf{x} \sqcap \mathbf{y}$ implies $\mathbf{y} \prec \mathbf{x} \sqcup \mathbf{y}$.
An important property of such lattices is they are graded by \emph{submodular} functions, that is $f:\L\to\mathbb{R}$ satisfying,
\begin{equation}\label{eq:submodular}
    f(\mathbf{x} \sqcap \mathbf{y}) + f(\mathbf{x} \sqcup \mathbf{y}) \leq f(\mathbf{x}) + f(\mathbf{y}).
\end{equation}

For two posets $\mathscr{P} = (\mathcal{X},\sqsubseteq)$ and $\mathscr{Q} = (\mathcal{Y},\rightarroweq)$, $\varphi:\mathscr{P} \to \mathscr{Q}$ is an \emph{order-preserving mapping} if and only if $\mathbf{x} \sqsubseteq \mathbf{y}$ implies $\varphi(\mathbf{x}) \rightarroweq \varphi(\mathbf{y})$ for all $\mathbf{x},\mathbf{y}\in\mathscr{P}$.
A mapping between lattices is \emph{join-preserving} if $\varphi(\mathbf{x} \sqcup \mathbf{y}) = \varphi(\mathbf{x})\sqcup \varphi(\mathbf{y})$.
A \emph{meet-preserving} mapping is defined similarly.
When a mapping is both join and meet-preserving, then it is a \emph{lattice-homomorphism}.
An injective lattice-homomorphism is a \emph{lattice-embedding}.
Observe, the domain of a lattice-embedding is isomorphic to a sublattice of the codomain.
A Galois connection is a type of duality defined by two order-preserving mappings.
\begin{definition}[Galois Connection]
    Let $\mathscr{P} = (\mathcal{X},\sqsubseteq)$ and $\mathscr{Q} = (\mathcal{Y},\rightarroweq)$ be reflexive posets.
    A \emph{Galois connection} consists of two order-preserving mappings $\varphi_*:\mathscr{P}\to\mathscr{Q}$ and $\varphi^*:\mathscr{Q} \to \mathscr{P}$ where,
    \[
        \varphi_*(\mathbf{x}) \rightarroweq \mathbf{y} \iff \mathbf{x} \sqsubseteq \varphi^*(\mathbf{y}),\quad\forall (\mathbf{x},\mathbf{y})\in\mathscr{P}\times\mathscr{Q}.
    \]
\end{definition}
In this context, one calls $\varphi_*$ the \emph{lower adjoint} and $\varphi^*$ the \emph{upper adjoint}, and we specify the connection via $\mathscr{P}\rightleftarrows\mathscr{Q}$.
Adjoints uniquely determine each other in any connection.
For example,
\begin{equation}\label{eq:adjoints}
    \varphi^*(\mathbf{y}) = \bigsqcup\{\mathbf{x}\in\mathscr{P}\mid \varphi_*(\mathbf{x})\rightarroweq \mathbf{y}\},
\end{equation}
when $\mathscr{P}$ and $\mathscr{Q}$ are lattices.
A mapping between lattices is join-preserving if and only if it is the lower adjoint of some connection.
The dual statement also holds.

\begin{lemma}[\cite{davey2002introduction}]\label{lem:gc-join}
    Let $\mathscr{P}$ and $\mathscr{Q}$ be posets with a map $\varphi:\mathscr{P}\to\mathscr{Q}$.
    Suppose further that both $\mathscr{P}$ and $\mathscr{Q}$ are finite lattices.
    Then, we have that $\varphi$ is (meet/join)-preserving if and only if $\varphi$ is the (upper/lower) adjoint of some Galois connection $\mathscr{P}\rightleftarrows\mathscr{Q}$.
\end{lemma}

Finally, one refers to $\varphi^*\circ\varphi_*$ as the \emph{closure composition}, as it always satisfies the axioms for a closure operator on a poset.
This is, a closure operator $\sigma:\mathscr{P}\to\mathscr{P}$ on $\mathscr{P} = (\mathcal{X},\sqsubseteq)$ is \emph{extensive}, i.e. $\mathbf{x}\sqsubseteq\sigma(\mathbf{x})$, \emph{monotone}, i.e. $\mathbf{x}\sqsubseteq \mathbf{y}$ implies $\sigma(\mathbf{x})\sqsubseteq\sigma(\mathbf{y})$, and \emph{idempotent}, i.e. $\sigma(\sigma(\mathbf{x})) = \sigma(\mathbf{x})$.
Furthermore, $\varphi_*\circ\varphi^*$ is the \emph{interior composition} as it is an \emph{interior operator}, a mapping $\tau:\mathscr{P}\to\mathscr{P}$ which is monotone and idempotent but \emph{deflationary}, i.e. $\mathbf{x} \sqsupseteq \tau(\mathbf{x})$, instead of extensive.

\subsection{Polymatroids}
Introduced by Edmonds \cite{edmonds2003submodular}, a \emph{polymatroid} is a polytope associated with a special submodular function.
Specifically, a \emph{polymatroid rank function} $\rho:2^\univ\to\mathbb{R}$ is a submodular set function which is \emph{normalized}, i.e.\ $\rho(\varnothing) = 0$, and \emph{monotone} in that $X \subseteq Y$ implies $\rho(X) \leq \rho(Y)$.
Here our submodular function is defined on a powerset, so Eq.~\ref{eq:submodular} is equivalent to the \emph{law of diminishing returns},
\[
    X \subseteq Y \implies (\forall z \notin Y)\;(\rho/X)(z) \geq (\rho/Y)(z),
\]
where $\rho/X$ is the \emph{contraction} of $\rho$ by $X$ given by $(\rho/X)(Y)\triangleq \rho(X \cup Y) - \rho(X).$
Note we abuse notation by letting $(\rho/X)(y) = 0$ for all $y \in X$.
A polymatroid is \emph{integral} if the codomain of its rank function is the integers. 
The \emph{span} operator $\sigma_\rho:2^\univ\to 2^\univ$ is defined as,
\begin{equation}\label{eq:pm-span-def}\sigma_\rho(X) \triangleq\{y \in \univ\mid \rho(X + y) = \rho(X)\},\end{equation}
i.e.\ those elements for which the marginal return $(\rho/X)(y) = 0$. 
The span is a closure operator, and so those sets $X \subseteq \univ$ such that $X = \sigma_\rho(X)$ are called $\rho$-\emph{closed}.
Ordering the $\rho$-closed sets by containment makes a lattice (Figures~\ref{fig:ubg-rep} and~\ref{fig:greedy-rep} give Hasse diagrams of such lattices) which we call $\L(\rho)$ for rank function $\rho$.
Closure operators on a powerset form \emph{Moore families}, and so the $\rho$-closed sets are closed under intersection (i.e. $\sigma_\rho(X)\sqcap\sigma_\rho(Y) = \sigma_\rho(X)\cap\sigma_\rho(Y)$ for all $X,Y\subseteq\univ$).

\subsection{Greedoids}
Let $\univ$ be an alphabet, and $\univ^*$ the free monoid (i.e.\ all possible finite sequences over $\univ$. 
Every $\alpha \in \univ^*$ can be written as $\alpha = x_1\ldots x_k$, where each $x_i \in \univ$,
so one refers to the elements of
$\univ^*$ as \emph{words} and $\univ$ as \emph{letters}. A collection of words $\glang \subseteq
\univ^*$ forms a \emph{language} over $\univ$ which is \emph{simple}
whenever no letter appears
twice in any word of the language. 
The \emph{support} of
a word $\alpha\in\univ^*$ is,
$$\widetilde{\alpha}\triangleq\{x\in\univ\mid(\exists \beta,\gamma\in\univ^*)\;\beta x \gamma = \alpha\},$$
i.e. the collection of letters forming $\alpha$, and we refer to the length of a
word via $|\alpha|$.
The longest words of $\glang$ are called \emph{basic}, and 
we denote the identity element of
$\univ^*$ by $\epsilon$. Thus, 
$\alpha\epsilon = \epsilon\alpha = \alpha$ for all $\alpha \in \univ^*$, meaning $\epsilon$ holds the privileged position of the \emph{empty word}.
A \emph{greedoid} follows.
\begin{definition}[Greedoid]\label{def:greedoid}
    A \emph{greedoid} $\glang$ over an alphabet $\univ$ is a simple language satisfying:
    \begin{enumerate}
        \item \emph{(Non-Empty) }The greedoid possesses at least one feasible word, i.e. $\glang \neq \varnothing$,
        \item \emph{(Hereditary) }$\alpha\beta \in \Lambda$ implies the prefix $\alpha$ is feasible, i.e. $\alpha\in\Lambda$,
        \item \emph{(Exchange) }For all $\alpha, \beta \in \Lambda$, if $|\alpha| > |\beta|$ then there exists $x\in\widetilde{\alpha}$ such that $\beta x \in \Lambda$.
    \end{enumerate}
\end{definition}

Note that non-emptiness and the hereditary axioms together imply that $\epsilon \in \glang$, and that the exchange axiom makes all basic words equal in length.
There is an equivalent definition in terms of \emph{accessible} set families making it obvious that greedoids generalize matroids by removing the need for downward closure, see either of \cite{bjorner1992introduction} or \cite{korte2012greedoids} for example.
However, we find the language definition more intuitive, and better suited to our proof techniques.

Because a greedoid encodes a set of constraints on valid ways to build a solution in a stepwise fashion, we call the words of a greedoid \emph{feasible}.
Hence, the familiar reader sees that $\glang$ is a matroid if and only if $\alpha \in \glang$ implies that every permutation of $\alpha$ is also feasible, i.e. $\widetilde{\alpha}! \subseteq \glang$ (see \cite{brylawski1988exchange} for more on this perspective).
A motivating example is \emph{undirected branching greedoids}.

\begin{example}[Undirected Branching Greedoid]\label{ex:undirected-branching-greedoid}
    Let $(V, E)$ be a connected graph with undirected edges, and select a root vertex $s \in V$.
    Let the alphabet be given by $E$, and the feasible words be simple $x_1\ldots x_k \in E^*$ with the following:
    For all $i \leq k$, the edges $\{x_1, \ldots, x_i\}$ induce a connected acyclic subgraph containing $s$ and $x_i$ is incident to the subgraph induced by $\{x_1,\ldots,x_{i-1}\}$.
    Verifying the hereditary and exchange axioms is left to the reader.
    Observe, the words correspond to feasible sequences of decisions made by Prim's algorithm, and cannot induce a matroid (in particular, no support without an edge incident to the root can be feasible) in contrast to Kruskal's algorithm.
\end{example}

The (greedoid) \emph{rank} $\rank(X)$ of $X\subseteq \univ$ is given by the length of the longest word in $\glang$ constructed using only letters from $X$.
The rank of the greedoid $r(\glang)$ is the length of a basic word.
A letter $x \in \univ$ is a \emph{loop} whenever there exists no feasible word containing it.
Such letters are deficient, and so a greedoid is \emph{normal} if its alphabet contains no loops.
Loops create busy notation for no benefit, so we sometimes assume a greedoid is normal.
The \emph{basis rank} of a set $X$ is the cardinality of the largest intersection between $X$ and the support of a basic word, i.e. $b(X)\triangleq \max_{\alpha\in\glang}|\widetilde{\alpha}\cap X|.$
The rank closure $\sigma_r:2^\univ\to 2^\univ$ is defined in the same way as a polymatroid span.
But this is not a formal closure operator as $\sigma_r$ lacks monotonicity.
The \emph{kernel closure} $\kappa(X)$ of a set $X$ is the union of feasible supports contained in $\sigma_{\rank}(X)$,
\begin{equation}\label{eq:flat support}
    \kappa(X) \triangleq \bigcup\left\{\widetilde{\alpha}\mid \alpha \in \glang\text{ and }\widetilde{\alpha} \subseteq \sigma_{\rank}(X)\right\}.
\end{equation}
Hence, for normal matroids one sees that the kernel closure and span operator are in agreement, i.e.\ $\kappa(X) = \sigma_{\rank}(X)$ for all $X \subseteq \univ$.
However, $\kappa$ is not extensive (or even monotone for greedoids without the interval property).
An exception is feasible supports, since $\alpha\in\glang$ implies $\widetilde{\alpha}\subseteq\sigma_r(\widetilde{\alpha})$.
This observation extends to any \emph{partial alphabet}, a union of feasible supports, as well.

\begin{figure}[t!]
    \centering
    \subfloat[Graph with root $s$.]
    {
        \scalebox{.8}{\input{tikz/udb-graph.tikz}}
    }
    \hfil
    \subfloat[Hasse diagram of flats.]
    {
        \scalebox{.8}{\input{tikz/udb-lattice-of-flats.tikz}}
    }
    \caption{A graph with root $s$ and the lattice of flats of the corresponding undirected branching greedoid. We see that $[ac]\sqcap[ad] = [a]$ because $ac \not\sim ad$, for example. Furthermore, the reader should note (or convince themselves) that the top flat of a greedoid is always given by the basic words. 
    }\label{fig:ubg}
\end{figure}

For $\alpha \in \glang$ the contraction minor $\glang/\alpha$ is the suffixes which can be concatenated onto $\alpha$ while maintaining feasibility in $\glang$,
\[
    \glang/\alpha \triangleq \{\beta \in \univ^*\mid \alpha\beta \in \glang\}.
\]
Contraction minors are greedoids, which is to say contraction preserves the non-emptiness, hereditary, and exchange axioms.
For another $\beta \in \glang$, $\alpha \sim \beta$ if $\glang/\alpha = \glang/\beta$.
A \emph{flat} is an equivalence class in $\glang/\mathord{\sim}$, and the equivalence class containing $\alpha$ is $[\alpha]$.
We say $\alpha$ \emph{spans} $F \in \glang/\mathord{\sim}$ if $\alpha \in F$ (equivalently, $[\alpha] = F$).
Flats make a partial order: For $\alpha,\beta\in\glang$, $[\alpha] \sqsubseteq [\beta]$ if and only if there exists $\gamma\in\glang/\alpha$ where $\alpha \gamma\in\glang$ and $\alpha \gamma \sim \beta$.
In some sense, this means that the flat $[\beta]$ is spanned by some number of words with prefixes spanning $[\alpha]$.
Moreover, $[\alpha]\prec[\beta]$ if and only if there exists $x\in\univ$ and $\gamma \in \univ^*$ such that $\gamma \sim \alpha$ and $\gamma x \sim \beta$.
To see this, examine Figure~\ref{fig:ubg} which shows the lattice of flats of an undirected branching greedoid, as described in Example~\ref{ex:undirected-branching-greedoid}.
Finally, let $\glang/[\alpha] = \glang/\alpha$ so that we can use the notation $\glang/F = \glang/\alpha$ for any choice of $\alpha \in F$ and flat $F\in\glang/\mathord{\sim}$.

Note that the flats of a contraction minor $\glang/F'$ are in bijective correspondence with some \emph{principal filter}, that is a filter with a single minimum element, of the poset of flats of $\glang$.
In many of our proofs it will be useful to make use of this.
Specifically, let $F\mapsto F/F'$ be the described bijection, which is to say,
\[
    F/F' \triangleq \{\alpha\in\glang/F'\mid (\exists \beta \in F')\;\beta\alpha \in F\}.
\]
When referring to $F/F'$, we tacitly assume $F'\sqsubseteq F$.
This is, $F/F'$ is undefined whenever $F'\not\sqsubseteq F$.

Call a simple word $\alpha'\in\univ^*$ a \emph{subword} of $\alpha\in\glang$ if and only if
$\widetilde{\alpha}' \subseteq \widetilde{\alpha}$ and the letters of $\alpha'$ occur in an order agreeing with $\alpha$ (i.e. $x$ appears before $y$ in $\alpha'$ implies $x$ appears before $y$ in $\alpha$).
Then, the interval property is defined by the guarantee that a shorter feasible word can always be augmented by a subword of a longer feasible word to achieve length at least that of the longer word.
\begin{definition}[Interval Property]
    Let $\glang$ be a greedoid.
    Then, $\glang$ has the \emph{interval property} if for all $\alpha ,\beta \in \glang$ with $|\beta| > |\alpha|$, there exists a subword $\beta'$ of $\beta$ of length $|\beta'| \geq |\beta| - |\alpha|$ with $\alpha \beta'\in\glang$.
\end{definition}
All greedoids under our consideration will satisfy the interval property.
For interval greedoids, the flats and the image of the kernel closure operator are in bijective correspondence because the interval property ensures that a set $X \in \kappa\left(2^\univ\right)$ if and only if the set of feasible words with supports contained in $X$ form a flat \cite{korte2012greedoids}.
For this reason we call a set in the $\kappa$-image a \emph{flat support}, and (as we always assume interval property) abuse notation by letting,
$$\kappa[\alpha] = \kappa(F) = \bigcup \{\widetilde{\beta}\mid \beta \in F\}.$$
This means that $\kappa(F)$ is the flat support of $F$.
Then, depending on context (that being whether its operand is a set of letters or a flat) either $\kappa$ maps from $2^\univ$ to $2^\univ$ like Eq.~\ref{eq:flat support}, or from $\glang/\mathord{\sim}$ to $2^\univ$ in the way described above.
Furthermore, due to this bijection we introduce another abuse of notation by letting $\kappa^{-1}(X)$ give the unique flat $F\in\glang/\mathord{\sim}$ such that $\kappa(F) = \kappa(X)$, which can be defined more precisely as,
\begin{equation}\label{eq:kappa-inv}
    \kappa^{-1}(X) \triangleq \bigsqcup \big\{[\alpha] \bigm\vert \alpha \in \glang\cap X^*\big\}.
\end{equation}
As justification, note $\kappa^{-1}\circ\kappa$ is the identity mapping when applied to $\glang/\mathord{\sim}$.
Another useful property is that the flats of an interval greedoid are \emph{cospanning}, which is to say that $\alpha \sim \beta$ if and only if $\sigma_r(\widetilde{\alpha}) = \sigma_r(\widetilde{\beta})$.
Correspondingly, since the \emph{continuations} of a feasible $\alpha$, that is those letters $x \in \univ$ such that $\alpha x$ is feasible, are the complement of the span $\sigma_r(\widetilde{\alpha})$ it follows that all words in a flat have the same set of continuations as well (in the presence of the interval property).
Thus, we use a notation similar to that enjoyed by the flat supports by letting $\Gamma[\alpha] \triangleq \univ\setminus\sigma_r(\widetilde{\alpha})$ be the continuations of any feasible $\alpha$, and $\Gamma(F) = \Gamma[\alpha]$ for $\alpha \in F$.

Finally, the flats of an interval greedoid form a semimodular lattice \cite{korte2012greedoids}.
This lattice, which we refer to by $\L(\glang)\triangleq (\glang/\mathord{\sim},\sqsubseteq)$, is graded by the rank of the words spanning a flat.
In light of this, we let $r(F) = r(\widetilde{\alpha})$ for any $\alpha \in F$.
Moreover, the lattice of flats is isomorphic to the flat supports ordered by containment \cite{korte2012greedoids}, i.e. $\left(\kappa\left(2^\univ\right),\subseteq\right)$.
We often use this fact implicitly.
A nice consequence is,
\begin{align*}
    \kappa(F\sqcup F') = \kappa(\kappa(F)\cup \kappa(F')),&&\text{and},&&\kappa(F\sqcap F') = \kappa(\kappa(F)\cap\kappa(F')),
\end{align*}
for all $F,F'\in\glang/\mathord{\sim}$.
To see the former, first observe $\kappa(\kappa(F)\cup\kappa(F')) = \kappa(F)\sqcup \kappa(F')$ as one can verify $\kappa(\kappa(F)\cup\kappa(F'))$ is the least flat support containing $\kappa(F)\cup\kappa(F')$.
Then, $\kappa(F)\sqcup\kappa(F') = \kappa(F\sqcup F')$ by the isomorphism.
The latter $\kappa(F\sqcap F') = \kappa(\kappa(F)\cap\kappa(F'))$ follows similarly.

\section{Polymatroid Representations}\label{sec:pm-greedoid}

A polymatroid greedoid $\glang$ is such that there exists a polymatroid rank function $\rho:2^\univ\to\mathbb{R}$ encoding the feasible words.
Specifically, we call $\rho$ a \emph{representation} if and only if,
\begin{equation}\label{eq:rep}
    \glang = \{x_1\ldots x_k\in\univ^*\mid (\forall i \in \{1,\ldots,k\})\;\rho(\{x_1,\ldots,x_i\}) = i\}.
\end{equation}
Korte and Lov{\'a}sz define the \emph{polymatroid greedoid} as greedoids possessing integral representations.
However, we generalize this.
By inspecting Eq.~\ref{eq:rep}, we see that a representation is in partial agreement with the greedoid rank, since they are equal on feasible supports.
However, more interesting is the representation maintains a submodular structure even among sets which \emph{cannot} be permuted into a feasible word.
The class of polymatroid greedoid is broad, including matroids, poset antimatroids, undirected branching greedoids, and certain ordered geometries.
See \cite{korte1985polymatroid} or Ch. 7 of \cite{korte2012greedoids}.

\begin{example}[Undirected Branching Greedoid]\label{ex:ubg-rep}
     Let $(V, E)$ be a graph and $s \in V$ a root.
    Then, for a set of edges $X \subseteq E$ let $\rho(X)$ give the number of vertices in $V - s$ covered by $X$.
    Then $\rho$ is monotone and normalized, and one can verify the law of diminishing returns.
    Let $\glang$ be the undirected branching greedoid.
    Fix $\alpha\in \glang$.
    The length $|\alpha|$ is precisely the number of vertices covered by $\widetilde{\alpha}$ not equal to $s$, so $\rho(\widetilde{\alpha}) = |\alpha|$.
    As $\alpha$ was generic, $x \in \Gamma[\alpha]$ implies $\left(\rho/\widetilde{\alpha}\right)=1$.
    For $x \notin\Gamma[\alpha]$, we see that both or none of the endpoints of $x$ are covered, so $\left(\rho/\widetilde{\alpha}\right)\in\{0,2\}$.
    See Table~\ref{table:ubg-rep} and Figure~\ref{fig:ubg-rep}.
\end{example}

\begin{figure}
\begin{floatrow}
\capbtabbox[0.5\textwidth]{%
    \centering
    \scalebox{.8}{\begin{tabular}{|c c || c  c|} 
     \hline
     $X$ & $\rho(X)$ & $X$ & $\rho(X)$ \\
     \hline\hline
     $\varnothing$ & 0 & $\{b, c\}$ & 3 \\ 
     \hline
     $\{a\}$ & 1 & $\{b, d\}$ & 3 \\
     \hline
     $\{b\}$ & 2 & $\{c, d\}$ & 3 \\
     \hline
     $\{c\}$ & 2 & $\{a,b,c\}$ & 3 \\
     \hline
     $\{d\}$ & 2 & $\{a,b,d\}$ & 3 \\  
     \hline
     $\{a,b\}$ & 3 & $\{a,c,d\}$ & 3 \\
     \hline
     $\{a,c\}$ & 2 & $\{b,c,d\}$ & 3 \\
     \hline
     $\{a,d\}$ & 2 & $\{a,b,c,d\}$ & 3 \\
     \hline
\end{tabular}}
    \vspace{.325\baselineskip}
}{%
    \caption{Values of the representation from Example~\ref{ex:ubg-rep} applied to the greedoid in Figure~\ref{fig:ubg}.}\label{table:ubg-rep}%
}\ffigbox[0.5\textwidth]{%
    \centering
    \scalebox{.8}{\input{tikz/pm-spans.tikz}}
}{%
    \caption{The Hasse diagram of the closed sets\\of $\rho$ from Table~\ref{table:ubg-rep}, ordered by containment.}\label{fig:ubg-rep}%
}
\end{floatrow}
\end{figure}

Polymatroid greedoids belong to a subclass of interval greedoids called \emph{local poset greedoids} \cite{korte1985polymatroid}.
Informally, this property enforces a kind of local distributive lattice structure on the feasible supports by ensuring that union and intersection of feasible supports contained in a larger feasible support can always be permuted into feasible words.

\begin{definition}[Local Poset Property \cite{korte1985polymatroid}]\label{def:local-poset}
    Let $\glang$ be a greedoid. Then, $\glang$ possesses the \emph{local poset property} if and only if for all $\alpha,\beta\in\glang$, if there exists $\gamma\in\glang$ whose support contains $\widetilde{\alpha}$ and $\widetilde{\beta}$, then the union $\widetilde{\alpha}\cup\widetilde{\beta}$ and the intersection $\widetilde{\alpha}\cap\widetilde{\beta}$ are also feasible supports, i.e., 
\begin{equation}\label{eq:local-poset}
    \left(\exists \gamma \in \glang\right)\; \widetilde{\alpha}\cup \widetilde{\beta} \subseteq \widetilde{\gamma} \implies \left(\widetilde{\alpha}\cup \widetilde{\beta}\right)! \cap \glang \neq \varnothing\text{ and }\left(\widetilde{\alpha}\cap \widetilde{\beta}\right)!\cap \glang \neq \varnothing.
\end{equation}
\end{definition}

Though assumed by \cite{korte1985polymatroid}, our work does not require a representation be integral.
Furthermore, the proof showing that polymatroid greedoids are local poset greedoids in \cite{korte1985polymatroid} does not require integrality.
We verify this in Appendix~\ref{app:deferred-lp} for completeness.
\begin{restatable}[\cite{korte1985polymatroid}]{proposition}{intervalprop}\label{prop:representation->local-poset} 
    If a greedoid $\glang$ has any polymatroid representation $\rho:2^\univ\to\mathbb{R}$, then it possesses the local poset property.
\end{restatable}

Our first main result in Section~\ref{sec:characterization} requires the construction of a polymatroid representation, however doing so under the constraint that the representation is challenging.
Though integrality doesn't seem necessary for defining polymatroid greedoids, general continuous representations do lose important properties.
This includes relationships with the flat supports proven in Main Result~\ref{mr:strong-pm-greedoid}, as well as the optimism property we introduce in Section~\ref{subsec:optimistic}.
So, in the next section we take an axiomatic approach towards defining what makes a representation ``useful,'' and we expand our definition of a polymatroid greedoid to be those possessing \emph{aligned representations}.
Following this, we introduce \emph{optimism}, a useful weakening of the monotonicity of a matroid's span, in Section~\ref{subsec:optimistic} and the Forking Lemma in Section~\ref{subsec:forking}.
Both will be pervasive in our analyses in Sections~\ref{sec:characterization} and~\ref{sec:galois}.
\subsection{Alignment}\label{subsec:alignment}

Because the lattice of closed sets encodes the combinatorial information of a polymatroid, we define alignment in terms of a special order-preserving mapping between $\L(\glang)$ and $\L(\rho)$.

\begin{definition}[Aligned Representation]\label{def:aligned}
    Let $\glang$ be a greedoid and $\rho:2^\univ\to\mathbb{R}$ a polymatroid rank function, with $\L(\glang)$ and $\L(\rho)$ being the respective lattices of flats and closed sets.
    Then, $\L(\glang)$ and $\L(\rho)$ are \emph{aligned} if and only if there exists an order-preserving $\varphi: \L(\glang) \to \L(\rho)$ satisfying,
    \begin{enumerate}
        \item \emph{(Agreement)} $\rho(\widetilde{\alpha}) = \rho(\varphi[\alpha])$ for all feasible words $\alpha \in \glang$,
        \item \emph{(Inclusion)} $\widetilde{\alpha}\subseteq\varphi[\alpha]$ for all feasible words $\alpha \in \glang$,
        \item \emph{(Cover-Preservation)} $F \prec F' \implies \varphi(F) \prec \varphi(F')$ for all greedoid flats $F, F'\in\glang/\mathord{\sim}$.
    \end{enumerate}
   %
\end{definition}

These axioms are a minimal set of conditions necessary for the combinatorial structures of a greedoid and a representation to be compatible in a way that isn't spurious.
They can be intuited and justified without reference to our forthcoming results:
\emph{Agreement} and \emph{inclusion} both ensure the representation maintains information about the greedoid's flats within the polymatroid's closure system.
Specifically, agreement preserves the rank information, while inclusion enforces the identities of letters in $\rho$-closed sets map back to the supports of the $\glang$-flats. 
\emph{Cover-preservation} ensures chains of $\L(\rho)$ can be associated with feasible words, much like how feasible words encode sets of chains in $\L(\glang)$.
Such makes $\L(\rho)$ reflect the feasibility structure whenever $\L(\rho)$ and $\L(\glang)$ are aligned.

If $\rho$ is a representation where $\L(\glang)$ and $\L(\rho)$ are aligned we call it an \emph{aligned representation}.
Using results from \cite{korte1985polymatroid}, verifying integral representations are \emph{always} aligned is straightforward.

\begin{proposition}\label{prop:int->aligned}
    Let $\glang$ be a greedoid and $\rho$ a polymatroid rank function.
    Suppose that $\rho$ is a representation of $\glang$.
    If $\rho$ is integral, then it is an aligned representation.
\end{proposition}

This requires the following statement about the partial alphabets of the integral polymatroid greedoids examined in \cite{korte1985polymatroid}.

\begin{lemma}[Corollary 4.4 of \cite{korte1985polymatroid}]\label{lem:pa-rank}
    Let $\glang$ be a greedoid with integral representation $\rho$.
    Then, for all partial alphabets $X$, $\rho(X) = r(X)$.
\end{lemma}

\begin{proof}[Proof of Proposition~\ref{prop:int->aligned}]
    We show that $\varphi=\sigma_\rho\circ\kappa$ is a suitable mapping witnessing alignment.
    Because a flat support is a partial alphabet, Lemma~\ref{lem:pa-rank} implies $\rho(\kappa[\alpha])= r(\kappa[\alpha])$ for all $\alpha \in \glang$.
    Since $r(\kappa[\alpha]) = r[\alpha]$, agreement follows.
    One also sees cover-preservation by integrality.
    Finally, inclusion follows by our choice of $\varphi$.
\end{proof}

\subsection{Optimism}\label{subsec:optimistic}

Suppose one wishes to build a basic word $x_1\ldots x_r \in\glang$ in an iterative fashion.
A subtle property of a normal matroid is that a letter $y$ is a continuation of the empty word $\epsilon$, and remains so until the first prefix containing $y$ in its span.
However, this dynamic is lost in the greedoid setting.
Because of its usefulness, we introduce a weakening of this dynamic being every letter is a continuation at \emph{some} point when building basic words step by step.
We call this \emph{optimism}.
\begin{definition}[Optimistic Greedoid]\label{def:optimism}
    A greedoid $\glang$ is \emph{optimistic} if and only if for all $y\in\univ$ which are not loops, and basic words $x_1\ldots x_r \in \glang$, there exists some index $i$ such that $y \in \Gamma[x_1\ldots x_i]$.
    Furthermore, $\glang$ is \emph{strongly optimistic} if and only if there exists an index $i$ such that,
    \[
        y \in \Gamma[x_1\ldots x_j] \iff i\leq j\text{ and }x \notin\kappa[x_1\ldots x_j]. 
    \]
\end{definition}
For Sections~\ref{sec:characterization} and~\ref{sec:galois}, it is helpful to observe that strong optimism makes $x\notin\kappa[\alpha]$ imply the existence of $\beta x\in\glang/\alpha$.
Furthermore, strong optimism also makes $F\prec F'$ and $x \in \kappa(F')\setminus\kappa(F)$ imply $x\in\Gamma(F)$.
Now we show that aligned representations imply optimism.

\begin{theorem}\label{thm:pm->optimism}
    Let $\glang$ be a greedoid possessing an aligned representation $\rho$.
    Then, $\glang$ is optimistic.
\end{theorem}

\begin{proof}
    Let $\varphi:\L(\glang)\to\L(\rho)$ be the mapping satisfying the alignment axioms, and
    fix a basic word $x_1\ldots x_R \in \glang$ and non-loop $y\in\univ$.
    Note that inclusion implies $\kappa[\alpha]\subseteq\varphi[\alpha]$ for all $\alpha \in \glang$.
    So, because there is a feasible word containing $y$ in its support it follows that $y \in \varphi[x_1\ldots x_R]$.
    Moreover, agreement makes $\rho(\varphi[\epsilon]) = 0$, and so $\varphi[\epsilon]$ can only contain loops of the greedoid.
    This makes $y \in \varphi[x_1\ldots x_R]\setminus\varphi[\epsilon]$, and as cover-preservation implies,
    \[
        \varphi[\epsilon] \prec \varphi[x_1]\prec\varphi[x_1x_2]\prec\ldots\prec\varphi[x_1\ldots x_R],
    \]
    it follows that there exists an index $i \in \{1,\ldots, R\}$ such that $y \in \varphi[x_1\ldots x_i]\setminus\varphi[x_1\ldots x_{i-1}]$.
    By cover-preservation, we see that $\sigma_\rho(\varphi[x_1\ldots x_{i-1}] + y) = \varphi[x_1\ldots x_i],$ which makes,
    \[
        (\rho/\{x_1,\ldots,x_{i-1})(y) = (\rho/\varphi[x_1\ldots x_{i-1}])(y) = (\rho/\varphi[x_1\ldots x_{i-1}])(x_i) = (\rho/\{x_1\ldots x_{i-1}\})(x_i) = 1.
    \]
    Thus $y \in \Gamma[x_1 \ldots x_{i-1}]$.
    As our choices were arbitrary, optimism follows.
\end{proof}

\begin{remark}
    Though they do not use our terminology, optimism was proven for integral representations in Corollary 4.17 of \cite{korte1985polymatroid}.
    We gave a shorter argument for the more general setting.
\end{remark}
\begin{remark}
    A general continuous representation does not imply optimism.
\end{remark}

The relationships between optimistic greedoids and other classes in the greedoid taxonomy are shown in Figure~\ref{fig:tax}, and proven in Appendix~\ref{app:optimism} and~\ref{app:local-augmentation}.
In particular, we show that every strong-exchange greedoid is optimistic.
Furthermore, we use a trimmed matroid construction to show the existence of an optimistic interval greedoid without the local poset property, and adapt a construction from \cite{korte1988intersection} to make an optimistic local poset greedoid which cannot possess any aligned representation.
Describing polymatroid greedoids by greedoid properties which exist independently of any polymatroid representation is an important open problem and, by the latter construction, optimism and the local poset property cannot do so alone.

\subsection{The Forking Lemma}\label{subsec:forking}
\begin{figure}[t!]
    \centering
    \scalebox{.8}{\input{tikz/forking}}
    \caption{Above is a sublattice of $\L(\glang)$ satisfying the premise described in the Forking Lemma. We draw the covering relation by solid lines, labelled by $x$ as it is a continuation advancing words in the lower neighbor to words in the upper neighbor, and denote a flat lying below another in a way which may or may not be covering by the dotted lines. Then, $F$ and $F'$ are such that the letter $x$ is a continuation of the meet $F\sqcap F' = [\mu]$, with $[\mu x]$ lying below $F'$ but not $F$. Then, $x$ must be a continuation of $F$ as a consequence. By semimodularity $F \sqcup [\mu x]$ covers $F$, so $F \prec \kappa^{-1}(\kappa(F) + x)$ (what we just found) and $\kappa^{-1}(\kappa(F) + x) \sqsubseteq F\sqcup [\mu x]$ imply $\kappa^{-1}(\kappa(F) + x) = F\sqcup [\mu x]$.}\label{fig:forking}
\end{figure}


Now we present our Forking Lemma for interval greedoids.
This lemma will be frequently applied in Sections~\ref{sec:characterization} and~\ref{sec:galois} for obtaining our main results.
Furthermore, towards the end of the section we use our lemma to show that optimistic interval greedoids satisfy a stronger form of Definition~\ref{def:optimism}.

\begin{lemma}[Forking Lemma]\label{lem:forking}
    Let $\glang$ be a greedoid with the interval property, and select two flats $F, F'\in\glang/\mathord{\sim}$.
    Suppose that $\mu \in F\sqcap F'$.
    Then, $x \in \Gamma(F\sqcap F')$ and $[\mu x] \sqsubseteq F'$ implies $x \in \Gamma(F)$.
\end{lemma}

We imagine the continuation $x\in\Gamma(F\sqcap F')$ as creating a ``fork'' in that it leads to a flat beneath $F'$ but not $F$ whenever $F\neq F'$, see Figure~\ref{fig:forking}.
In some sense, the lemma thus shows that concatenation with $x$ is partially commutative.
Compare these guarantees to the exchange and interval properties:
Suppose $\alpha, \beta \in \glang$ with $|\alpha| < |\beta|$.
Then, the exchange and interval properties guarantee \emph{existence} of $x \in\widetilde{\beta}\setminus \widetilde{\alpha}$ with $x\in\Gamma[\alpha]$, but cannot guarantee a \emph{specific} identity of $x$.
In contrast, the Forking Lemma can be applied to witness a specific identity of such a continuation $x$.

To prove this result, we induct over the difference in grade between $F$ and the meet $F\sqcap F'$.
For this, we will require the following weaker result making a statement similar to the Forking Lemma, but only for words differing by one letter in the same location.
\begin{lemma}\label{lem:inner-forking}
    Let $\glang$ be a greedoid with the interval property, and $\mu \in \glang$ a feasible word.
    If there exists $x,y \in \Gamma[\mu]$ and $\alpha \in \univ^*$ with $\mu x \alpha$ and $\mu y \alpha$ feasible, then $\mu x \alpha \not\sim \mu y \alpha$ implies $\mu y \alpha x \in \glang$.
\end{lemma}
\begin{proof}
    Select $x,y \in \Gamma[\mu]$ and $\alpha \in \univ^*$ satisfying the premise.
    Note that $\mu x \alpha \not\sim \mu y \alpha$ ensures that $x \neq y$, making the task of showing $x \in \Gamma[\mu y \alpha]$ well-formed.
    To do so, first recall that two words in the same flat of an interval greedoid are cospanning.
    Hence, via negation we find $\Gamma[\mu x \alpha] \neq \Gamma[\mu y \alpha]$.
    First, suppose $\Gamma[\mu x \alpha] \setminus \Gamma[\mu y \alpha] \neq \varnothing$, and select a letter $z$ in this difference.
    Then because,
    \begin{equation}\label{eq:pidgeonhole}
        \widetilde{\mu x \alpha z} \setminus \widetilde{\mu y \alpha} = \{x, z\},
    \end{equation}
    and the words of $\glang$ are simple, the fact that $\mu x \alpha z \in \glang$ and $|\mu x \alpha z| > |\mu y \alpha|$ makes the lemma by way of the exchange axiom.
    More precisely, since $z \notin \Gamma[\mu y \alpha]$ and Eq.~\ref{eq:pidgeonhole} makes $x$ the only letter that can be concatenated onto $\mu y \alpha$ from $\widetilde{\mu x \alpha z}$ while maintaining feasibility, we conclude that $\mu y \alpha x \in \glang$.
    
    To finish, we must also verify that the case of $\Gamma[\mu x\alpha]\setminus \Gamma[\mu y \alpha] = \varnothing$ does not pose any problems.
    If this is true, then there exists $z \in \Gamma[\mu y \alpha]\setminus \Gamma[\mu x \alpha]$.
    Then, by symmetry the preceding arguments show that $\mu x \alpha y \in \glang$.
    Therefore, by simple exchange it must be the case that $\mu y \alpha x$ is feasible as well, because $|\mu x \alpha y| > |\mu y \alpha|$ and $\widetilde{\mu x \alpha y} \setminus \widetilde{\mu y \alpha} = \{x\}$.
\end{proof}
\begin{proof}[Proof of the Forking Lemma]
    Let $\mu \in \glang$ and $x \in \Gamma[\mu]$ be as described in Lemma~\ref{lem:forking}.
    Put more concisely, our assumptions make,
    $$F\sqcap F' = [\mu] \sqsubset [\mu x] \sqsubseteq F'.$$
    and from this, the reader should observe that the strict ordering relation between $[\mu]$ and $[\mu x]$ ensures $[\mu x] \not\sqsubseteq F$, thereby making $F \neq F'$ as well.
    Hence, there exists a letter $y \in \Gamma[\mu]$ and $\alpha \in \univ^*$ where $y \neq x$ and $\mu y \alpha \in F$.
    To prove the lemma, we will induct over the difference $r(F) - r[\mu]$.
    In the base case where $r(F) - r(F\sqcap F') = 1$, we see that $\alpha = \epsilon$.
    Then, because $\mu y \in F$ and $[\mu x] \not \sqsubseteq F$, it is the case that $\mu x \not\sim \mu y$.
    Therefore, by applying Lemma~\ref{lem:inner-forking} to $\mu x \epsilon$ and $\mu y \epsilon$ we find $\mu y \epsilon x = \mu y x$ is feasible, which makes the base case as this is equivalent to stating $x \in \Gamma(F)$.

    Now, recalling $\mu y \alpha \in F$, for the inductive step we make $r(F) - r[\mu] > 1$ by assuming $\alpha = \beta z$ where $z \in \univ$ is a letter and $\beta \in \univ^*$ is a simple word, possibly $\beta = \epsilon$.
    By the interval property there exists a subword of $\mu y \beta z$ of length at least $|\mu y \beta z| - |\mu x| = |\beta| + 1$ which can be concatenated onto $\mu x$ to make a feasible word.
    The resulting feasible word must be simple (no letters can be repeated), and so there are three possibilities for the identity of this observed subword:
    \begin{enumerate}
        \item The subword equals $y\beta$, and so $\mu x y \beta$ is feasible,
        \item The subword equals $\beta z$, and so $\mu x \beta z$ is feasible,
        \item There exists another subword $\beta'$ of $\beta$ resulting from the deletion of some single letter, which is to say $|\beta'| = |\beta| -1$, such that $\mu x y \beta' z$ is feasible.
    \end{enumerate}
    Suppose that either (1.) or (3.) is true.
    Then, in either of these cases the hereditary axiom makes $\mu xy$ feasible, and thus $\mu y x$ is feasible as well by exchange from $\mu xy$ to $\mu y$.
    Furthermore, $[\mu x] \not\sqsubseteq F$ implies $[\mu y x] \not\sqsubseteq F$ as well (since $[\mu x] \prec [\mu y x]$ by $y \in \Gamma[\mu x]$), and so $[\mu y] \sqsubseteq F$ makes $F \sqcap [\mu yx] = [\mu y]$ by construction.
    As a consequence,
    \[
        r(F) - r(F \sqcap [\mu y x]) = r(F) - r([\mu y]) = r(F) - r([\mu]) + 1.
    \]
    Therefore, by applying the inductive hypothesis to $F$ and $[\mu y x]$, we witness $x \in \Gamma(F)$.
    The only remaining case is (2.), where $\mu x \beta z \in \glang$.
    Here, we simply apply Lemma~\ref{lem:inner-forking} to $\mu x \beta z$ and $\mu y \beta z$ to find $\mu y \beta z x \in \glang$.
    Then, because this case makes $\mu y \beta z \in F$ we conclude that $x \in \Gamma(F)$.
\end{proof}

By applying the Forking Lemma over an inductive structure, we can prove a stronger statement which concerns words rather than single letters.
Namely, repeated applications show the existence of words in $F\sqcup F'/F$ which are constructed from subwords of words in $F/F\sqcap F'$.

\begin{proposition}[Amplified Forking Lemma]\label{prop:amplified-forking}
    Let $\glang$ be a greedoid with the interval property, and select two flats $F, F'\in\glang/\mathord{\sim}$.
    For all $\alpha \in F/F\sqcap F'$, there exists a subword $\alpha'$ of $\alpha$ with $\alpha' \in F\sqcup F'/F$.
\end{proposition}
\begin{proof}
    Induct over $|\alpha|$.
    If $\alpha = \epsilon$, then $F\sqsubseteq F'$.
    So $F=F\sqcap F'$ and $F' = F\sqcup F'$, making $\epsilon \in F/F\sqcap F'$ and $\epsilon \in F\sqcup F'/F'$.
    Now for the inductive step, let $\alpha = \beta x$ and $F_0$ be the flat covered by $F$ which is formed by concatenating any word spanning $F\sqcap F'$ with $\beta$.
    By hypothesis, there exists a subword $\beta'$ of $\beta$ such that $\beta'\in F_0\sqcup F'/F'$.
    If $F_0\sqcup F' = F\sqcup F'$ then we are done,
    so assume $F_0\sqcup F' \neq F\sqcup F'$.
    This makes $F\not\sqsubseteq F_0\sqcup F'$, and so it follows that,
    \begin{equation}\label{eq:igiveup}
        F \sqcap (F_0\sqcup F') = F_0, 
    \end{equation}
    since $F_0 \prec F$, and $F_0$ lies below both $F$ and $F_0\sqcup F'$.
    By substitution $F \succ F \sqcap (F_0\sqcup F')$, so,
    \[
        F_0\sqcup F' \prec F \sqcup (F_0\sqcup F') = F\sqcup F',
    \]
    by semimodularity.
    Observe, because $F_0\sqcup F'$ is rank one less than $F\sqcup F'$, appending any letter in $\Gamma(F_0\sqcup F')\cap\kappa(F\sqcup F')$ to $\beta'$ leads to a word spanning $F\sqcup F'/F'$.
    Well, $\alpha = \beta x$ implies $x \in \kappa(F)$, and Eq.~\ref{eq:igiveup} shows $x$ is a continuation of the meet of $F$ and  $F_0\sqcup F'$ leading to a flat ($F$ to be precise) which is not lesser than $F_0\sqcup F'$.
    Then, by applying the Forking Lemma for $x \in \Gamma(F_0\sqcup F')$, we see that $\alpha'=\beta'x$ gives the result.
\end{proof}

The Forking Lemma combines nicely with optimism to show that flat supports witness the feasibility of eventually appending a letter $y$ onto some extension of a feasible word $\alpha$.
In particular, $y\notin\kappa[\alpha]$ implies the existence of $\beta \in \glang/\alpha$ with $\alpha\beta y$ feasible.
In many ways, this is a generalization of the matroid fact that points beyond the span of an independent set can be used to obtain a larger independent set, and
so is a pleasant example of polymatroid greedoids lying between matroids and general greedoids in terms of their axiomatic power.

\begin{proposition}\label{prop:strong-opt}
    Let $\glang$ be a greedoid with the interval property.
    If $\glang$ is optimistic, then for all $\alpha \in \glang$, $y \notin \kappa[\alpha]$ implies the existence of $\beta \in \glang/\alpha$ with $y\in \widetilde{\beta}$.
\end{proposition}
\begin{proof}
    By optimism, at least one of the following are true:
    There exists an $\alpha$-prefix $\alpha'$ such that $y \in \Gamma[\alpha']$, or there exists $\beta\in\glang/\alpha$ with $y \in \widetilde{\beta}$.
    The latter is our target statement, so suppose the former is true.
    Because $y \notin \kappa[\alpha]$ implies $[\alpha'y] \neq [\alpha]$, it is the case that $[\alpha'y]\sqcap [\alpha] \sqsubset [\alpha'y]$.
    Observe, $[\alpha']$ lies below both $[\alpha'y]$ and $[\alpha]$.
    So, since $[\alpha'y]$ covers $[\alpha']$, it follows that $[\alpha'] = [\alpha' y] \sqcap [\alpha]$.
    Hence, we may apply the Forking lemma to $[\alpha'y]$ and $[\alpha]$ to see $y \in \Gamma[\alpha]$.
    Thus $y\in\glang/\alpha$, so the word $\beta = y$ satisfies the proposition.
\end{proof}

\section{Polymatroid Greedoids and the Greedy Algorithm}\label{sec:characterization}
\begin{algorithm*}[t]
    \DontPrintSemicolon
    \small
    $\alpha\gets\epsilon$\;
    \While{$\Gamma[\alpha]\neq\varnothing$}
    {
        $x \gets \argmin_{y\in\Gamma[\alpha]}w(y)$\;
        $\alpha\gets\alpha x$\;
    }
    \;
    \Return $\alpha$
    \caption{The Greedy Algorithm}\label{alg:greedy}
\end{algorithm*}

Let $w: \univ\to\mathbb{R}$ be a function assigning weights to the alphabet.
Suppose we wish to solve a linear optimization over the basic words of the greedoid as defined by this weight function:
\begin{equation}\label{prog:lin-greedoid}
    \def\arraystretch{1.25}
    \begin{array}{lll}
        \text{minimize}  & \sum_{x \in\widetilde{\alpha}}w(x),\quad & \\
        \text{subject to}& \alpha\in\glang, \quad & \\
                         & r[\alpha] = r(\glang). \quad &
    \end{array}
\end{equation}
This is of course an abstraction of matroid optimization, where the greedoid can be used to encode more complex constraints.
Consider the greedy algorithm:
Maintain a solution $\alpha$, initially $\alpha=\epsilon$, and grow over $r(\glang)$ rounds.
In each round, if $\Gamma[\alpha]=\varnothing$ then terminate.
Else, select a continuation $x$ of minimum weight and repeat on $\alpha x$.
Korte and Lovász showed that the correctness of this approach is logically equivalent to the greedoid being sufficiently ``matroid-like.''

\begin{theorem}[\cite{korte1984greedoids}]\label{thm:greedy}
    Let $\glang$ be a greedoid with rank and basis ranks functions $r,b:2^\univ\to\mathbb{Z}$.
    The greedy algorithm correctly outputs the basic word minimizing $\alpha\mapsto \sum_{x\in\widetilde{\alpha}}w(x)$ for every weighting $w:\univ\to\mathbb{R}$ if and only if $b$ is a matroid rank function such that every $r$-closed set is $b$-closed.
\end{theorem}

In \cite{korte1984greedoids,goetschel1986linear} this is shown to be equivalent to greedoids possessing the \emph{strong-exchange property}, which we use in Appendix~\ref{app:greedy-minors}.
For this reason, we refer to greedoids for which the greedy algorithm is guaranteed to correctly solve Program~\ref{prog:lin-greedoid} \emph{strong exchange greedoids}.
Theorem~\ref{thm:greedy} shows such greedoids have a structure associated with a polymatroid rank function, i.e. its basis rank.
We use this to construct an aligned representation of any local poset greedoid with the strong-exchange property.
This will give the following:

\begin{theorem}\label{thm:lp-greedy=>pm-greedoid}
    Let $\glang$ be a greedoid with the local poset property,
    and suppose the greedy algorithm correctly solves any linear optimization over the basic words.
    Then, there exists an aligned polymatroid representation $\rho:2^\univ\to\mathbb{R}$ of $\glang$.
\end{theorem}

Our proposed representation $\greedyrep$ follows: First, define $b_F:2^\univ\to\mathbb{Z}$ to be the basis rank of the minor $\glang/F$, i.e. $b_F(X)\triangleq \max_{\alpha\in\glang/F}|\widetilde{\alpha}\cap X|.$
Then, $\greedyrep$ is defined as the constrained optimization,
\begin{align*}
    \greedyrep(X) \triangleq \min_{F\in\glang/\mathord{\sim}}\left\{\psi_F(X)\mid (\exists \alpha \in F)\;\widetilde{\alpha}\subseteq X\right\},&&\text{and},&&\psi_F(X) \triangleq r(F) + \left(1+\frac{1}{r(\glang)+1}\right)\cdot b_F(X).
\end{align*}
Observe that the minimization is constrained to flats containing a word supported by letters in $X$.
This assertion and its resulting structure will be required in our technical lemmas.
In this spirit, refer to any flat containing a word supported by $X$ as \emph{$\greedyrep(X)$-feasible}, and call a $\greedyrep(X)$-feasible flat minimizing $F\mapsto \psi_F(X)$ among other feasible flats \emph{$\greedyrep(X)$-optimal}.
Now, see that the cost function $F\mapsto\psi_F(X)$ is the flat's rank ``penalized'' by one's ability to use $X$ to construct $(\glang/F)$-feasible words.
This penalty is encoded by the minor's basis rank, so is minimal (i.e. zero) when $X\subseteq\kappa(F)$ and maximal (i.e $r(\glang)-r(F)$) when $X$ contains the support of a $(\glang/F)$-basic word.
We weigh $b_F(X)$ by a factor $1+\frac{1}{r(\glang)+1}$ greater than one to introduce marginal returns which are not equal to one whenever evaluated against the support of a feasible word $\alpha\in\glang$ on a letter $x\notin\Gamma[\alpha]$.
Furthermore,
\begin{equation}\label{eq:floor-psi}
    \lfloor \psi_F(X) \rfloor = r(F) + b_F(X),\quad\forall (X,F) \in 2^\univ\times\glang/\mathord{\sim},
\end{equation}
as it is always true that $b_F(X) \leq r(\glang)$.

\begin{remark}
    Take note that the basis rank by itself is not strong enough to constitute a representation (whenever it is submodular) since $\alpha\in\glang$ and $(b/\widetilde{\alpha})(x)=1$ does not imply $x \in\Gamma[\alpha]$.
\end{remark}

\begin{figure}
\begin{floatrow}
\capbtabbox[0.5\textwidth]{%
    \centering
    \renewcommand{\arraystretch}{1.25}
    \scalebox{.8}{\begin{tabular}{|c c || c  c|} 
     \hline
     $X$ & $\rho(X)$ & $X$ & $\rho(X)$ \\
     \hline\hline
     $\varnothing$ & 0 & $\{b, c\}$ & $2 + \frac{2}{r(\glang) + 1}$\\[2pt]
     \hline
     $\{a\}$ & 1 & $\{b, d\}$ & $2 + \frac{2}{r(\glang) + 1}$ \\[2pt]
     \hline
     $\{b\}$ & $1 + \frac{1}{r(\glang) + 1}$ & $\{c, d\}$ & $2 + \frac{2}{r(\glang) + 1}$\\[2pt]
     \hline
     $\{c\}$ & $1 + \frac{1}{r(\glang) + 1}$ & $\{a,b,c\}$ & 3 \\[2pt]
     \hline
     $\{d\}$ & $1 + \frac{1}{r(\glang) + 1}$ & $\{a,b,d\}$ & 3 \\[2pt] 
     \hline
     $\{a,b\}$ & $2 + \frac{1}{r(\glang) + 1}$ & $\{a,c,d\}$ & 3 \\[2pt]
     \hline
     $\{a,c\}$ & 2 & $\{b,c,d\}$ & $2 + \frac{2}{r(\glang) + 1}$ \\[2pt]
     \hline
     $\{a,d\}$ & 2 & $\{a,b,c,d\}$ & 3 \\[2pt]
     \hline
\end{tabular}}
}{%
    \caption{Values of $\greedyrep$ when applied to the greedoid in Figure~\ref{fig:ubg}.}\label{table:greedy-rep}%
}\ffigbox[0.5\textwidth]{%
    \centering
    \scalebox{.8}{\input{tikz/gr-spans.tikz}}
}{%
    \caption{The Hasse diagram of the closed sets\\of $\greedyrep$ from Table~\ref{table:greedy-rep}, ordered by containment.}\label{fig:greedy-rep}%
}
\end{floatrow}
\end{figure}

As a concrete example, we show the values taken by $\greedyrep$ over the undirected branching greedoid from Figure~\ref{fig:ubg} and its lattice of closed sets in Figure~\ref{fig:greedy-rep}.
The reader should compare with Figure~\ref{fig:ubg-rep}.
Observe that $\greedyrep(\widetilde{\alpha}) = \psi_{[\alpha]}(\widetilde{\alpha})$ for every feasible word $\alpha$ in this example, so $x\in\Gamma[\alpha]$ implies $(\greedyrep/\widetilde{\alpha})(x) = 1$ while $x \notin\Gamma[\alpha]$ implies $(\greedyrep/\widetilde{\alpha})(x)$ equals 0 or $1+\frac{1}{r(\glang)+1}$.
This relationship is true in general.

\begin{proposition}\label{prop:greedy-aligned-rep}
    Let $\glang$ be a greedoid with the interval property.
    Suppose the greedy algorithm correctly solves any linear optimization on $\glang$.
    Then, for all feasible words $\alpha \in \glang$ and $x\in\univ$,
    \[
        x \in \Gamma[\alpha] \iff \left(\greedyrep/\widetilde{\alpha}\right)(x) = 1,
    \]
    and $\left(\greedyrep/\widetilde{\alpha}\right)(x) \in \{0\}\cup[1,\infty)$. 
\end{proposition}
\begin{proof}
    First, we show $\greedyrep(\widetilde{\alpha}) = \psi_{[\alpha]}(\widetilde{\alpha})$ for all $\alpha \in \glang$.
    Select any $\greedyrep(\widetilde{\alpha})$-optimal flat, i.e. some $F\in\glang/\mathord{\sim}$ achieving $\greedyrep(\widetilde{\alpha}) = \psi_F(\widetilde{\alpha})$ such that there exists $\beta \in F$ with $\widetilde{\beta}\subseteq\widetilde{\alpha}$.
    Through the exchange axiom there is some $\gamma \in (\widetilde{\alpha}\setminus\widetilde{\beta})!$ such that $\gamma\in\glang/F$.
    Hence,
    $$b_F(\widetilde{\alpha})\geq|\widetilde{\gamma}\cap\widetilde{\alpha}| = |\gamma| = |\alpha| - r(F).$$
    Applying this to the definition of $\psi_F$ and rearranging terms shows,
    $$\psi_{F}(\widetilde{\alpha}) \geq r(F) + \left(1+\frac{1}{r(\glang) + 1}\right)\cdot\left(|\alpha| - r(F)\right) = |\alpha| + \frac{1}{r(\glang) + 1}\cdot\left(|\alpha| - r(F)\right)\geq |\alpha|.$$
    But $\psi_{[\alpha]}(\widetilde{\alpha}) = |\alpha|$, so our assumptions require $\psi_F(\widetilde{\alpha}) \leq |\alpha|$ also hold.
    Thus, $\greedyrep(\widetilde{\alpha}) = \psi_{F}(\widetilde{\alpha}) = |\alpha|$.

    Now fix $x \in \univ$.
    Notice, if $x \in\Gamma[\alpha]$ then $\alpha x$ is feasible.
    So, our preceding argument shows the desired $x \in\Gamma[\alpha]$ implies $\greedyrep(\widetilde{\alpha x}) - \greedyrep(\widetilde{\alpha}) = 1$.
    Then, for the other direction assume $x \notin\Gamma[\alpha]$.
    Once again select a $\greedyrep(\widetilde{\alpha})$-optimal flat $F$, and let $\beta\in F$ be any word with $\widetilde{\beta}\subseteq\widetilde{\alpha}$.
    Observe $x\notin\widetilde{\beta}$, as otherwise the word $\beta'x$ is a $\beta$-prefix for some subword $\beta'$ of $\beta$.
    Then, applying the Forking Lemma to $[\beta' x]$ and $[\alpha]$ (take note that $[\beta']\sqsubseteq [\alpha]$ and $[\beta'x]\succ[\beta']$ implies $[\beta']=[\alpha]\sqcap[\beta'x]$) shows $x\in \Gamma[\alpha]$, a contradiction.
    Hence $\widetilde{\beta}\subseteq\widetilde{\alpha}$, which implies $[\beta]\sqsubseteq [\alpha]$.
    Then, recalling that our assumptions about the greedy algorithm ensure that $b$ satisfies the law of diminishing returns, by way of Lemma~\ref{prop:bf} one finds $\left(b_{F}/\widetilde{\alpha}\right)(x)=\left(b_{[\beta]}/\widetilde{\alpha}\right)(x) \geq \left(b_{[\alpha]}/\widetilde{\alpha}\right)(x)$.
    So by combining this with $[\alpha]$ being $\greedyrep(\widetilde{\alpha})$-optimal,
    \begin{multline*}
        \psi_{F}(\widetilde{\alpha}+ x) = r(F) + \left(1 + \frac{1}{r(\glang) + 1}\right)\cdot b_F(\widetilde{\alpha}) +\left(1 + \frac{1}{r(\glang) + 1}\right)\cdot \left(b_{F}/\widetilde{\alpha}\right)(x),\\\geq \psi_{[\alpha]}(\widetilde{\alpha}) + \left(1 + \frac{1}{r(\glang) + 1}\right)\cdot \left(b_{F}/\widetilde{\alpha}\right)(x) = \psi_{[\alpha]}(\widetilde{\alpha}+x).
    \end{multline*}
    Thus $[\alpha]$ is $\greedyrep(\widetilde{\alpha}+x)$-optimal as well.
    So, as $\left(\greedyrep/\widetilde{\alpha}\right)(x) = \psi_{[\alpha]}(\widetilde{\alpha}+x) - \psi_{[\alpha]}(\widetilde{\alpha})$, we see $x\notin\Gamma[\alpha]$ implies $\left(\greedyrep/\widetilde{\alpha}\right)(x) \neq 1$ by,
    \begin{equation*}
        \psi_{[\alpha]}(\widetilde{\alpha}+x) - \psi_{[\alpha]}(\widetilde{\alpha}) = \begin{cases}
            1 + \frac{1}{r(\glang) + 1},\quad&\text{if }b_{[\alpha]}(x) \neq 0,\\
            0,&\text{otherwise}.
        \end{cases}
    \end{equation*}
    And we've already seen $\left(\greedyrep/\widetilde{\alpha}\right)(x) = 1$ whenever $x \in\Gamma[\alpha]$, so $\left(\greedyrep/\widetilde{\alpha}\right)(x) \in \{0\}\cup[1,\infty)$ holds too.
\end{proof}

The previous proposition will be used to show $\greedyrep$ is an aligned representation whenever $\greedyrep$ is a polymatroid rank function.
Since $\greedyrep$ is clearly normalized, we need only show monotonicity and submodularity.
To do so, we begin in Section~\ref{subsec:bf-rank} by proving facts about the function $b_F$.
Here, our main insight is that the basis rank $b_F$ of the minor $\glang/F$ can be equated with the (polymatroidal) contractions $b/\kappa(F)$.
With this, we will leverage the law of diminishing returns when analyzing $\greedyrep$.
This analysis occurs in Section~\ref{subsec:greedy-aligned-rep}, where we complete the proof of Theorem~\ref{thm:lp-greedy=>pm-greedoid} by verifying $\greedyrep$ is monotone and submodular.

\subsection{The Basis Rank of the Contraction Minors of Interval Greedoids}\label{subsec:bf-rank}
Our main result of this section is the following, equating the values of $b_F$ to those of $b/\kappa(F)$.
\begin{proposition}\label{prop:bf}
    Let $\glang$ be a greedoid with the interval property.
    Then, $b_F = b/\kappa(F)$ everywhere.
\end{proposition}
\begin{proof}
    Select an arbitrary set of letters $X\subseteq \univ$.
    It should be clear that $b_F(X) \leq (b/\kappa(F))(X)$.
    For the reverse inequality, take a subset $Y\subseteq X$ such that $(b/\kappa(F))(X) = (b/\kappa(F))(Y) = |Y|$.
    Such a $Y$ must exist:
    Because $b$ is subcardinal, there must be a sequence $y_1,\ldots, y_k\in X$ wsuch that $k = (b/\kappa(F))(X)$ and $(b/\kappa(F) \cup \{y_1,\ldots,y_{i-1}\})(x_i) = 1$ for all $i \in \{1,\ldots, k\}$.
    Hence,
    \[
        k = (b/\kappa(F))(\{y_1,\ldots,y_k\}) = (b/\kappa(F))(X),
    \]
    and so we can let $Y = \{y_1,\ldots, y_k\}$.
    Now, notice that by subcardinality we must have $(b/\kappa(F))(Y-z) = |Y|-1$ for all $z \in Y$.
    Hence, for any $z \in Y$,
    \[
        b(\kappa(F)\cup Y) - b(\kappa(F)\cup Y - z) = 1,
    \]
    and so $z$ lies in the support of any feasible word $\alpha \in \glang$ such that $|\widetilde{\alpha} \cap (\kappa(F)\cup Y)| = b(\kappa(F) \cup Y)$.
    However, as our choice of $z\in Y$ was arbitrary, this means,
    \[
        \alpha\in\glang\text{ and }|\widetilde{\alpha} \cap (\kappa(F)\cup Y)| = b(\kappa(F) \cup Y) \implies Y \subseteq \widetilde{\alpha}.
    \]

    With that, select a basic word $\beta \in \glang$ achieving $|\widetilde{\beta} \cap (\kappa(F)\cup Y)| = b(\kappa(F) \cup Y)$.
    Then, from $Y\subseteq \widetilde{\beta}$ and $(b/\kappa(F))(Y) = |Y|$ we observe,
    \begin{equation}\label{eq:2345r}
        |\widetilde{\beta}\cap \kappa(F)| = b(\kappa(F)\cup Y) - |Y| =b(\kappa(F)\cup Y) - (b/\kappa(F))(Y) =  b(\kappa(F)) = r(F).
    \end{equation}
    Or put in other terms, there are exactly $r(F)$ letters from $\kappa(F)$ in the support of $\beta$.
    Well, for any $\gamma \in F$, it follows by exchange that there is a simple word $\pi \in (\widetilde{\beta} \setminus \kappa(F))^*$ which can be appended onto $\gamma$ to make a basic word $\gamma\pi \in \glang$.
    However, the support of $\pi$ cannot overlap with $\kappa(F)$, as $\pi \in \glang/\gamma=\glang/F$ by construction.
    By this and Eq.~\ref{eq:2345r}, there are exactly $r(\glang)-r(F)$ to choose from $\widetilde{\beta}\setminus\kappa(F)$ to form $\pi$.
    But, because $|\gamma| = r(F)$ and $\gamma\pi$ is basic we also have that $|\pi| = r(\glang)-r(F)$.
    Therefore $\widetilde{\pi} = \widetilde{\beta}\setminus\kappa(F)$, which implies $Y \subseteq \widetilde{\pi}$.
    Then by recalling $\pi \in  \glang/F$ and $Y\subseteq X$,
    \[
        b_F(X) \geq b_F(Y) \geq |\widetilde{\pi} \cap Y| = |Y| = (b/\kappa(F))(X),
    \]
    and we conclude $b_F(X) = (b/\kappa(F))(X)$.
\end{proof}

By definition, this implies $b_F(X) = b(\kappa(F)\cup X) - b(\kappa(F))$.
But, because $\kappa(F)$ is a partial alphabet, $b(\kappa(F)) = r(\kappa(F))$.
This produces a helpful corollary.
\begin{corollary}\label{cor:bf-eq}
    Let $\glang$ be a greedoid with the interval property.
    Then, for all subsets $X\subseteq\univ$ and flats $F\in\glang/\mathord{\sim}$, we have $b_{F}(X) = b(\kappa(F)\cup X) - r(F)$.
\end{corollary}

\subsection{An Aligned Representation of Greedily Optimizable Local Poset Greedoids}\label{subsec:greedy-aligned-rep}

Recall, Proposition~\ref{prop:greedy-aligned-rep} shows $\greedyrep$ can be used to encode the feasible words of a greedoid à la polymatroid representation.
We must verify $\greedyrep$ is a polymatroid rank function.
A key challenge is that because $\greedyrep$ is defined as a constrained optimization over $\L(\glang)$, it is not easily evaluated on arbitrary sets of letters.
And so, as a preview, we develop technical lemmas for identifying $\greedyrep(X)$-optimal flats.

First, recall that $\lfloor\psi_F(X)\rfloor = r(F)+b_F(X)$, so it follows that any $\greedyrep(X)$-optimal flat minimizes $r(F)+b_F(X)$ among other $\greedyrep(X)$-feasible flats.
In the following two lemmas, we identify that a $\greedyrep(X)$-optimal flat actually minimizes $r(F)+b_F(X)$ among \emph{all} flats, and that the minimizers of this function form a sublattice of $\L(\glang)$.
This latter fact will provide an algebraic structure with which we can observe important properties of $\greedyrep(X)$-optimal flats.

\begin{lemma}\label{lem:opt-min-floor}
    Let $\glang$ be a greedoid with the interval property, and $X\subseteq \univ$ any collection of letters.
    Then, $\lfloor\greedyrep(X)\rfloor = \min\left\{r(F)+b_F(X)\mid (\exists \alpha \in F)\;\widetilde{\alpha}\subseteq X\right\}$.
\end{lemma}
\begin{proof}
    Recall $\lfloor\psi_F(X)\rfloor = r(F)+b_F(X)$ for all $F\in\glang/\mathord{\sim}$.
    It follows then that $r(F)+b_F(X) < r(F')+b_{F'}(X)$ implies $\psi_F(X) < \psi_{F'}(X)$.
    Hence, we need only show,
    \begin{equation}\label{eq:fjfkll}
        \{F\in\glang/\mathord{\sim}\mid(\exists\alpha\in F)\;\widetilde{\alpha}\subseteq X\}\cap\left(\argmin_{F\in\glang/\mathord{\sim}}r(F)+b_F(X)\right) \neq \varnothing.
    \end{equation}
    Well, via Corollary~\ref{cor:bf-eq} we have $b(\kappa(F)\cup X) = r(F) + b_F(X)$.
    Furthermore, monotonicity of the basis rank makes $b(X) \leq b(\kappa(F)\cup X)$, and for the bottom flat $[\epsilon]$ it should be clear that $r[\epsilon] + b_{[\epsilon]}(X) = b(X)$.
    Combining these observations together shows,
    \[
        r[\epsilon] + b_{[\epsilon]}(X) \leq r(F) + b_F(X),\quad\forall F\in\glang/\mathord{\sim},
    \]
    and so because the empty word trivially satisfies $\widetilde{\epsilon}\subseteq X$, the bottom flat $[\epsilon]$ witness Eq.~\ref{eq:fjfkll}.
\end{proof}

\begin{lemma}\label{lem:floor-psi-min-lattice}
    Let $\glang$ be a greedoid with the interval property, and $X\subseteq \univ$ any collection of letters.
    Suppose the greedy algorithm correctly solves any linear optimization over the basic words of $\glang$.
    Then, the set of flats minimizing $F\mapsto r(F) + b_F(X)$ forms a sublattice of $\L(\glang)$.
\end{lemma}
\begin{proof}
    This essentially follows by showing $F\mapsto r(F) + b_F(X)$ is a submodular function over $\L(\glang)$.
    Once submodularity is shown, a standard technique for showing that the minimizers form a distributive lattice can be applied (see \cite{queyranne1998minimizing} for a working example over a ring of sets).

    Select two flats $F,F'\in\glang/\mathord{\sim}$.
    To show submodularity, first observe because $\kappa(F)\cup\kappa(F')$ is a partial alphabet, it follows that $b(\kappa(F)\cup\kappa(F')) = r(\kappa(F)\cup\kappa(F'))$.
    However, $r(Y) = r(\kappa(Y))$ for all $Y\subseteq \univ$, and $\kappa(\kappa(F)\cup\kappa(F')) = \kappa(F\sqcup F')$ since $\kappa(\kappa(F)\cup\kappa(F'))$ is the join of $\kappa(F)$ and $\kappa(F')$ in the lattice of flat supports.
    Thus,
    \begin{equation}\label{eq:783473}
        \left(b/\kappa(F)\cup\kappa(F')\right)(\kappa(F\sqcup F')) = 0.
    \end{equation}
    Now, our assumptions make $b$ a matroid rank function.
    We leverage this with Corollary~\ref{cor:bf-eq},
    \begin{equation*} 
        \def\arraystretch{1.25}
        \begin{array}{lll}
            r(F) \!\!\!\!&+\; r(F') + b_F(X) + b_{F'}(X) = b(\kappa(F)\cup X) + b(\kappa(F') \cup X),\quad&\text{by Corollary~\ref{cor:bf-eq}},\\
                         &\geq b((\kappa(F)\cap\kappa(F'))\cup X) + b(\kappa(F)\cup\kappa(F')\cup X),\quad&\text{by submodularity},\\
                         &\geq b((\kappa(F)\cap\kappa(F'))\cup X) + b(\kappa(F\sqcup F')\cup X),\quad&\text{by Eq.~\ref{eq:783473}},\\
                                                      &= b_{F\sqcap F'}((\kappa(F)\cap\kappa(F'))\cup X) + b_{F\sqcup F'}(X) + r(F\sqcap F') + r(F\sqcap F'),\quad&\text{by Corollary~\ref{cor:bf-eq}},\\
                                                      &\geq b_{F\sqcap F'}(X) + b_{F\sqcup F'}(X) + r(F\sqcap F') + r(F\sqcup F'),\quad&\text{by monotonicity}.
        \end{array}
    \end{equation*}
    So we have submodularity over $\L(\glang)$.
    It is easy now to show that the minimizers form a sublattice: Suppose $F$ and $F'$ additionally satisfy,
    \begin{equation}\label{eq:min-assump}
        r(F) + b_F(X) = r(F') + b_{F'}(X)\text{ and }(\forall F'' \in \glang/\mathord{\sim})\;r(F) + b_F(X) \leq r(F'') + b_{F''}(X). 
    \end{equation}
    Well, by the averaging principle,
    \[
        \min\{r(F\sqcap F') + b_{F\sqcap F'}(X), r(F\sqcup F') + b_{F\sqcup F'}(X)\} \leq r(F) + b_F(X),
    \]
    which implies $\min\{r(F\sqcap F') + b_{F\sqcap F'}(X), r(F\sqcup F') + b_{F\sqcup F'}(X)\} = r(F) + b_F(X)$ by Eq.~\ref{eq:min-assump}.
    Therefore, to maintain the submodular law,
    \begin{equation*}
        \max\{r(F\sqcap F') + b_{F\sqcap F'}(X), r(F\sqcup F') + b_{F\sqcup F'}(X)\} \leq r(F') + b_{F'}(X).
    \end{equation*} 
    And so, $F\sqcap F'$ and $F\sqcup F'$ are also minimizers.
\end{proof}

Now we can make a number of observations regarding $\greedyrep(X)$-optimal flats.
These include uniqueness and monotonicity, among other properties.

\begin{lemma}\label{lem:big-stick}
    Let $\glang$ be a greedoid with the interval property, and fix $X,Y\subseteq \univ$ and $z\in\univ$.
    Suppose the greedy algorithm correctly solves any linear optimization over the basic words of $\glang$,
    and assume $F_X,F_{X+z}$, and $F_Y$ are $\greedyrep(X)$, $\greedyrep(X+z)$, and $\greedyrep(Y)$-optimal flats respectively.
    The following are true:
    \begin{enumerate}
        \item $F_X$ is the unique $\greedyrep(X)$-optimal flat,
        \item If $X\subseteq Y$, then $F_X \sqsubseteq F_Y$,
        \item $X\cap \Gamma(F_X) = \varnothing$,
        \item If $F_X \sqsubset F_{X+z}$, then there exists $\alpha \in (X\setminus\kappa(F_X))^*$ such that $z\alpha \in F_{X+z}/F_X$.
    \end{enumerate}
\end{lemma}
\begin{proof}
    Now, for 1., let $F_X'$ be another $\greedyrep(X)$-optimal flat.
    Lemma~\ref{lem:opt-min-floor} implies that both $F_X$ and $F_X'$ minimize $F\mapsto\lfloor\psi_F(X)\rfloor$, and so
    through Lemma~\ref{lem:floor-psi-min-lattice} we observe the join $F_X\sqcup F_X'$ minimizes this function as well.
    Without loss of generality, we assume either $F_X$ and $F_X'$ are incomparable or $F_X$ and $F_X'$ can be ordered as $F_X\sqsubseteq F_X'$.
    Hence, as some $\alpha\in F_X$ and $\beta \in F_X'$ such that $\widetilde{\alpha}\cup\widetilde{\beta}\subseteq X$ are guaranteed to exist by $\greedyrep(X)$-feasibility, we apply the Amplified Forking Lemma to witness the existence of a subword $\beta'$ of $\beta$ such that $\beta'\in F_X\sqcup F_X'/F_X$.
    Thus, $\alpha\beta'\in F_X\sqsubseteq F_X'$ shows that the join is $\greedyrep(X)$-feasible as well.
    Now, because $b_{F_X}(X) \geq |\widetilde{\beta'\gamma}\cap X|$ and $\widetilde{\beta}'\cap\widetilde{\gamma} = \varnothing$ for every $\gamma \in F_X\sqcup F_X'$, by letting selecting such a $\gamma$ satisfy $b_{F_X\sqcup F_X'}(X) =|\widetilde{\gamma}\cap X|$ (and recalling $\widetilde{\beta}'\subseteq\widetilde{\beta}\subseteq X$),
    \begin{equation}\label{eq:57ty}
        b_{F_X}(X) - b_{F_X\sqcup F_X'}(X) \geq |\widetilde{\beta'\gamma}\cap X| - |\widetilde{\gamma}\cap X| = |\beta'|.
    \end{equation}
    Therefore,
    \[
        \psi_{F_X}(X) - \psi_{F_X\sqcup F_X'}(X) = \lfloor\psi_{F_X}(X)\rfloor - \lfloor\psi_{F_X\sqcup F_X'}(X)\rfloor + \frac{b_{F_X}(X) - b_{F_X\sqcup F_X'}(X)}{r(\glang) + 1} \geq \frac{|\beta'|}{r(\glang) + 1}.
    \]
    Optimality requires this equal zero, forcing $|\beta'| = 0$.
    Thus $F_X = F_X\sqcup F_X'$, so $F_X$ and $F_X'$ are comparable as $F_X'\sqsubseteq F_X$.
    But, as comparability implies $F_X\sqsubseteq F_X'$ by earlier assumptions, $F_X = F_X'$.

    Next up we examine 2.
    Let $\beta \in F_X$ with $\widetilde{\beta}\subseteq X$.
    One witnesses a subword $\beta'$ of $\beta$ with $\beta'\in F_X\sqcup F_Y/F_Y$ by way of the Amplified Forking Lemma.
    Select $\alpha \in F_Y$ with $\widetilde{\alpha}\subseteq Y$.
    Then it is clear that $\alpha\beta'\in F_X\sqcup F_Y$ and $\widetilde{\alpha\beta'} \in F_X\sqcup F_Y$, thereby making the join $\greedyrep(Y)$-feasible.
    By repeating the same logic underlying Eq.~\ref{eq:57ty}, one can observe the bound,
    $$b_{F_Y}(Y) - b_{F_X\sqcup F_Y}(Y) \geq |\beta'| = r(F_X\sqcup F_Y) - r(F_Y).$$
    However, we also have $b_{F_Y}(Y) - b_{F_X\sqcup F_Y}(Y) \leq r(F_X\sqcup F_Y) - r(F_Y)$:
    Take a $\gamma \in \glang/F_Y$ with $b_{F_Y}(Y) = |\widetilde{\gamma}\cap Y|$.
    By the exchange axiom, there is a word $\gamma'\in \glang/F_X\sqcup F_Y$ of length $r(\glang) - r(F_X\sqcup F_Y)$ supported on a subset of $\widetilde{\gamma}$.
    This implies there are at most $r(F_X\sqcup F_Y) - r(F_Y)$ letters in $\widetilde{\gamma}\cap Y$ which are not contained in $\widetilde{\gamma}'\cap Y$.
    Hence,
    \[
        r(F_X\sqcup F_Y) - r(F_Y) \geq |\widetilde{\gamma}\cap Y| - |\widetilde{\gamma}'\cap Y| \geq b_{F_Y}(Y) - b_{F_X\sqcup F_Y}(Y).
    \]
    So, we've identified $b_{F_Y}(Y) - b_{F_X\sqcup F_Y}(Y) = r(F_X\sqcup F_Y) - r(F_Y)$.
    This makes,
    \[
        \psi_{F_X\sqcup F_Y}(Y) - \psi_{F_Y}(Y) = r(F_X\sqcup F_Y) - r(F_Y) + \left(1+\frac{1}{r(\glang) + 1}\right)\cdot\left(r(F_Y) - r(F_X\sqcup F_Y)\right) \leq 0,
    \]
    and so $F_X\sqcup F_Y$ is also $\greedyrep(Y)$-optimal.
    By 1., it can only be that $F_Y = F_X\sqcup F_Y$, making $F_X\sqsubseteq F_Y$.

    Now, we'll show 3. and 4. together.
    For 4. assume $F_X\sqsubset F_{X+z}$.
    Let $\beta \in F_{X+z}$ be such that $\widetilde{\beta}\subseteq X+z$.
    By the Amplified Forking Lemma there is a subword $\beta'$ of $\beta$ such that $\beta'\in F_{X+z}/F_X$.
    Notice that 3. and 4. follow if $\Gamma(F_X) \cap X = \varnothing$.
    Specifically, this is a restatement of 3., and it would force $\beta' = z\alpha$ for some $\widetilde{\alpha}\subseteq X$ by pigeonhole principle.
    So by way of contradiction, assume there exists $y \in \Gamma(F_X) \cap X$.
    For convenience, let $F_0 \triangleq \kappa^{-1}(\kappa(F_X)+y)$
    Notice, by fixing $\gamma \in \glang/F_0$ with $b_{F_0}(X) = |\widetilde{\gamma}\cap X|$, since $y\gamma \in F_X$ we have,
    \[
        b_{F_X}(X) - b_{F_0}(X) \geq |\widetilde{y\gamma}\cap X| - |\widetilde{\gamma}\cap X| = 1.
    \]
    Thus,
    \begin{multline*}
        \psi_{F_X}(X) - \psi_{F_0}(X) = r(F_X) - r(F_0) + \left(1+\frac{1}{r(\glang)+1}\right)\cdot\left(b_{F_X}(X) - b_{F_0}(X)\right),\\\geq r(F_X) - (r(F_X) + 1) + \left(1+\frac{1}{r(\glang)+1}\right)=\frac{1}{r(\glang)+1},
    \end{multline*}
    and so $\psi_{F_0}(X) < \psi_{F_X}(X)$.
    It is easy to see that $F_0$ is $\greedyrep(X)$-feasible (since we can append $y$ onto the end of any word in $F_X$ to achieve a word spanning $F_0$, and $y\in X$).
    So we've found that $F_X$ is not $\greedyrep(X)$-optimal, a contradiction.
    As discussed this witnesses 3. and 4.,
    and so concludes the lemma.
\end{proof}

In what follows we will continue to refer to the $\greedyrep(X)$-optimal flat by $F_X$.
Lemma~\ref{lem:big-stick}.1 justifies this.
Now, with Lemma~\ref{lem:big-stick}.2 it is straightforward to show that $\greedyrep$ satisfies monotonicity.

\begin{proposition}\label{prop:gr-monotonicity}
    Let $\glang$ be a greedoid with the interval property.
    Suppose the greedy algorithm correctly solves any linear optimization over the basic words of $\glang$.
    Then, for all $X,Y\subseteq\univ$, $X\subseteq Y$ implies $\greedyrep(X) \leq \greedyrep(Y)$.
\end{proposition}
\begin{proof}
    First we show $b(Z) = b(Z\cap \kappa({F_Z})) + b_{F_Z}(Z)$ for arbitrary $Z\subseteq\univ$.
    Our assumptions make the basis rank $b$ submodular.
    So,
    \begin{equation}\label{eq:dddss}
        b(Z) = b(Z \cap \kappa({F_Z})) + (b/Z \cap \kappa({F_Z}))(Z).
    \end{equation}
    Because ${F_Z}$ is $\greedyrep(Z)$-feasible the existence of $\alpha \in Z^*$ such that $\alpha \in {F_Z}$ follows.
    This construction implies $(b/Z \cap \kappa({F_Z}))(Z) = (b/\kappa({F_Z}))(Z)$ since by the law of diminishing returns,
    \[
        (b/\widetilde{\alpha})(Z) \geq (b/Z \cap \kappa({F_Z}))(Z) \geq (b/\kappa({F_Z}))(Z) = (b/\widetilde{\alpha})(Z).
    \]
    Thus we see,
    \begin{equation}\label{eq:bZ=sum}
        b(Z) = b(Z\cap \kappa({F_Z})) + b_{F_Z}(Z),
    \end{equation}
    by substituting $(b/Z \cap \kappa({F_Z}))(Z) = (b/\kappa({F_Z}))(Z)$ into Eq.~\ref{eq:dddss}.
    Since $b(Z\cap\kappa(F_Z)) = r(F_Z)$ by way of $\greedyrep(Z)$-feasibility, we observe $\lfloor \psi_{F_Z}(Z)\rfloor = b(Z)$ as a consequence.
    Finally, if $F \in \glang/\mathord{\sim}$ is such that $F\sqsubseteq F_Z$ then applying this argument to the minor $\glang/F$ shows,
    \[
        F\sqsubseteq F_Z \implies b_F(Z) = b_F(Z\cap\kappa({F_Z}) + b_{F_Z}(Z).
    \]
    Take note that $F_X \sqsubseteq F_Y$ by Lemma~\ref{lem:big-stick}.2, so $b_{F_X}(Y) = b_{F_X}(Y\cap\kappa(F_Y)) + b_{F_Y}(Y)$.
    However, because $F_Y$ is $\greedyrep(Y)$-feasible, $b_{F_X}(Y\cap\kappa(F_Y)) = r(F_Y) - r(F_X)$.
    Therefore,
    \begin{equation}\label{eq:ifyfzf}
        \lfloor \psi_{F_X}(Y)\rfloor = r(F_X) + b_{F_X}(Y) = r(F_X) - r(F_X) + r(F_Y) + b_{F_Y}(Y) = \lfloor \psi_{F_Y}(Y) \rfloor,
    \end{equation}
    i.e. $\lfloor \psi_{F_X}(Y)\rfloor = \lfloor \psi_{F_Y}(Y) \rfloor$.
    This and Eq.~\ref{eq:bZ=sum} is used in what follows.

    Now first suppose $F_X = F_Y$.
    Then $\greedyrep(X) \leq \greedyrep(Y)$ follows from $X\subseteq Y$, $r(F_X) = r(F_Y)$, and $b_{F_X} = b_{F_Y}$.
    In the other case, where $F_X \neq F_Y$, we see $F_X \sqsubset F_Y$ by Lemma~\ref{lem:big-stick}.2.
    So there exists a nonempty $z\beta \in F_Y/F_X$ with $\widetilde{z\beta}\subseteq Y$.
    This makes $z \in \Gamma(F_X)$,
    and by way of Lemma~\ref{lem:big-stick}.3 we see $z \notin X$.
    Now, we claim that any $\beta \in F_X$ such that $\widetilde{\beta}\subseteq X$ is such that $|\beta| = r(X)$, which makes $r(X) = r(F_X)$.
    Suppose this weren't true.
    Let $\gamma\in\glang$ be a longest word with $\widetilde{\gamma}\subseteq X$.
    Then since $\widetilde{\gamma}$ is independent in the matroid induced by $b$, it follows by exchange that $b(\widetilde{\gamma}) + (b/\widetilde{\gamma})(X) = b(X)$.
    Hence, as $b(\widetilde{\gamma}) = r[\gamma]$ and $b/\widetilde{\gamma} = b/\kappa[\gamma]$, we see $\lfloor \psi_{[\gamma]}(X)\rfloor = \lfloor\psi_{F_X}(X)\rfloor$ by way of Eq.~\ref{eq:bZ=sum}.
    But $r[\gamma] > r(F_X)$ by our assumptions, and so $\psi_{[\gamma]}(X) < \psi_{F_X}(X)$ which contradicts the optimality of $F_X$.
    So we have shown $r(X) = r(F_X)$.
    Since $z \in \Gamma(F_X)$, this implies $z \notin \sigma_r(X)$.
    By Theorem~\ref{thm:greedy}, we have that $\sigma_r(X)$ is $b$-closed.
    Then $(b/\sigma_{r}(X))(z) = 1$, and so by $z \in Y$ and the law of diminishing returns (using $\kappa(F_X)\subseteq \sigma_r(X)$ and Proposition~\ref{prop:bf}) we observe $(b_{F_X}/X)(Y) > 0$.
    The upshot is that this implies $\psi_{F_X}(Y) > \psi_{F_X}(X)$.
    Well, because the increments of $\psi_{F_X}$ are greater than or equal to one, this also makes $\lfloor\psi_{F_X}(Y)\rfloor > \lfloor\psi_{F_X}(X)\rfloor$.
    Therefore, by applying Eq.~\ref{eq:ifyfzf} we observe,
    \[
        \lfloor \psi_{F_Y}(Y)\rfloor = \lfloor\psi_{F_X}(Y)\rfloor > \lfloor\psi_{F_X}(X)\rfloor.
    \]
    But $\lfloor \psi_{F_Y}(Y)\rfloor > \lfloor\psi_{F_X}(X)\rfloor$ implies $\psi_{F_Y}(Y) > \psi_{F_X}(X)$.
    So monotonicity follows in all cases.
\end{proof}

What remains is submodularity.
Our strategy will be to show the law of diminishing returns.
This analysis starts with the following technical lemma.

\begin{lemma}
    Let $\glang$ be a greedoid with the interval property, and fix $X\subseteq \univ$ and $z \in \univ$.
    Suppose the greedy algorithm correctly solves any linear optimization over the basic words of $\glang$.
   Then, $F_X\sqsubset F_{X+z}$ implies $r(F_X) + b_{F_X}(X) + 1 = r(F_{X+z}) + b_{F_{X+z}}(X+z)$.
\end{lemma}
\begin{proof}
    Because $F_{X+z}$ is $\greedyrep(X+z)$-optimal, we require $\lfloor\psi_{F_X}(X+z)\rfloor \geq \lfloor \psi_{F_{X+z}}(X+z)\rfloor$.
    $F_X\sqsubset F_{X+z}$ implies $\left(b_{F_X}/X\right)(z) = 1$, so plugging in $\lfloor\psi_{F_X}(X+z)\rfloor = r(F_X) + b_{F_X}(X+z)$ and $\lfloor\psi_{F_{X+z}}(X+z)\rfloor = r(F_{X+z}) + b_{F_{X+z}}(X+z)$ this makes,
    \begin{equation}\label{eq:floor-psi-x-xz}
        r(F_X) + b_{F_X}(X) + 1 \geq r(F_{X+z}) + b_{F_{X+z}}(X+z).
    \end{equation}
    Assume the inequality is strict.
    Then, $r(F_{X+z}) + b_{F_{X+z}}(X+z) \leq r(F_X) + b_{F_X}(X)$ as $\left(b_{F_X}/X\right)(z) = 1$.
    However, $\lfloor\psi_{F_{X+z}}(X+z)\rfloor \geq \lfloor\psi_{F_X}(X)\rfloor$ must hold to maintain the monotonicity of $\greedyrep$, so we find $r(F_{X+z}) + b_{F_{X+z}}(X+z) = r(F_X) + b_{F_X}(X)$, or equivalently $\lfloor \psi_{F_{X+z}}(X+z) \rfloor = \lfloor \psi_{F_X}(X)\rfloor$.
    Because $F_X\sqsubset F_{X+z}$, $r(F_X) < r(F_{X+z})$.
    This makes $b_{F_X}(X) > b_{F_{X+z}}(X+z)$.
    Hence,
    \[
        \greedyrep(X) = \lfloor \psi_{F_X}(X)\rfloor + \frac{b_{F_X}(X)}{r(\glang)+1} < \lfloor \psi_{F_{X+z}}(X+z) \rfloor + \frac{b_{F_{X+z}}(X+z)}{r(\glang)+1} = \greedyrep(X+z),
    \]
    which contradicts the monotone structure of $\greedyrep$ shown in Proposition~\ref{prop:gr-monotonicity}.
    And so we conclude the desired $r(F_X) + b_{F_X}(X) + 1 = r(F_{X+z}) + b_{F_{X+z}}(X+z)$.
\end{proof}

By rewriting terms, the prior lemma shows both $r(F_{X+z}) - r(F_X) + b_{F_{X+z}}(X+z) - b_{F_X}(X)$ and $r(F_{Y+z}) - r(F_Y) + b_{F_{Y+z}}(Y+z) - b_{F_Y}(Y)$ equal one.
This implies,
\begin{equation}\label{eq:key1}
    r(F_{X+z}) - r(F_X) + b_{F_{X+z}}(X+z) - b_{F_X}(X) = r(F_{Y+z}) - r(F_Y) + b_{F_{Y+z}}(Y+z) - b_{F_Y}(X),
\end{equation}
which through Lemma~\ref{lem:opt-min-floor} is equivalent to $\lfloor \psi_{F_{X+z}}(X+z)\rfloor - \lfloor \psi_{F_X}(X)\rfloor = \lfloor \psi_{F_{Y+z}}(Y+z)\rfloor - \lfloor \psi_{F_Y}(Y)\rfloor$.
Then, keeping in mind that,
\begin{align*}
    \left(\greedyrep/X\right)(z) = \lfloor \psi_{F_{X+z}}(X+z)\rfloor - \lfloor \psi_{F_X}(X)\rfloor + \frac{b_{F_{X+z}}(X+z) - b_{F_X}(X)}{r(\glang)+1},
\end{align*}
while a similar expression holds for $\left(\greedyrep/Y\right)(z)$,
the law of diminishing returns will follow if,
\begin{equation}\label{eq:key2}
    b_{F_{X+z}}(X+z) - b_{F_X}(X) \geq b_{F_{Y+z}}(Y+z) - b_{F_Y}(Y).
\end{equation}
Via Eq.~\ref{eq:key1}, this is implied if $r(F_{Y+z}) - r(F_Y) \geq r(F_{X+z}) - r(F_X)$.
This is our key insight, proven in what follows whenever $\left(\greedyrep/Y\right)(z)\neq 0$.
We remark that the coming argument is where we fully leverage our assumptions of the local poset property and the correctness of the greedy algorithm.

\begin{lemma}\label{lem:key-lemma}
    Let $\glang$ be a greedoid with the local poset property.
    Suppose the greedy algorithm correctly solves any linear optimization over the basic words of $\glang$.
    Then, for all $X,Y\subseteq\univ$ and $z \in \univ$ such that $X\subseteq Y$ and $z \notin Y$,
    $\left(\greedyrep/Y\right)(z)\neq 0$ implies $r(F_{Y+z}) - r(F_Y) \geq r(F_{X+z}) - r(F_X)$.
\end{lemma}
\begin{proof}
    Let $F_X \neq F_{X+z}$, as otherwise $r(F_{X+z}) - r(F_X) = 0$ and the lemma follows trivially.
    Then, by Lemma~\ref{lem:big-stick}.4 there exists $z\alpha\in F_{X+z}/F_X$ satisfying $\widetilde{\alpha}\subseteq X$.
    We prove the lemma by verifying $z\alpha \in \glang/F_Y$.
    To see why this suffices, observe that $z\alpha\in\glang/F_Y$ implies,
    \begin{equation*}
        \def\arraystretch{1.25}
        \begin{array}{lll}
            \kappa(F_Y)\cup\widetilde{z\alpha} &\subseteq \kappa\left(\kappa(F_Y)\cup\widetilde{z\alpha}\right),\quad&\text{because  $\kappa(F_Y)\cup\widetilde{z\alpha}$ is a partial alphabet},\\
                                               &\subseteq \kappa\left(\kappa(F_X)\cup\widetilde{z\alpha}\cup\kappa(F_Y)\right),\quad&\text{by Lemma~\ref{lem:big-stick}.2},\\
                                               &\subseteq \kappa\left(\kappa(F_{X+z})\cup\kappa(F_Y)\right),\quad&\text{by $z\alpha\in F_{X+z}/F_X$},\\
                                               &= \kappa(F_{X+z}\sqcup F_Y),
        \end{array}
    \end{equation*}
    and so because one can examine Eq.~\ref{eq:kappa-inv} to see $\kappa^{-1}:\left(2^\univ,\subseteq\right)\to\L(\glang)$ is order-preserving, we find $\kappa^{-1}(\kappa(F_Y)\cup\widetilde{z\alpha}) \sqsubseteq F_{X+z}\sqcup F_Y$.
    But by assuming $z\alpha\in\glang/F_Y$, it follows that,
    $$r(\kappa^{-1}(\kappa(F_Y)\cup\widetilde{z\alpha})) - r(F_Y) = |z\alpha|.$$
    Hence, by recalling that Lemma~\ref{lem:big-stick}.2 shows $F_{Y+z} \sqsupseteq F_{X+z} \sqcup F_Y$,
    \[
        r(F_{Y+z}) - r(F_Y) \geq |z\alpha| = r(F_{X+z})-r(F_X),
    \]
    and so the lemma follows upon verifying $z\alpha\in\glang/F_Y$.

    We'll show $z\alpha \in \glang/F_Y$ by proving that $z\alpha'\in\glang/F_Y$ for every $\alpha$-prefix $\alpha'$.
    Induct over $|\alpha'|$, and
    first examine the case $\alpha' = \epsilon$.
    Because $\left(\greedyrep/Y\right)(z)\neq 0$ by assumption, we find $\left(b_{F_Y}/Y\right)(z) = 1$.
    This makes $z \notin \kappa(F_Y)$, and so $\kappa^{-1}(\kappa(F_X) + z) \not\sqsubseteq F_Y$.
    Therefore, because our construction makes $z \in \Gamma(F_X)$, $z \in \Gamma(F_Y)$ follows by applying the Forking Lemma to $\kappa^{-1}(\kappa(F_X) + z)$ and $F_Y$.
    Now for the inductive step, let $\alpha' = \beta x$, and fix any $y_1\ldots y_k \in F_Y/F_X$.
    We first show,
    \begin{equation}\label{eq:cant-use-x}
        (\forall x' \in\widetilde{\beta x})(\nexists i \in\{1,\ldots, k\})\;y_1\ldots y_i x' \in\glang/F_X.
    \end{equation}
    Examine any $y' \in \univ$ and $i\in\{1,\ldots, k\}$ making $y_1\ldots y_i y' \in \glang/F_X$.
    There are two cases, given by the ordering of $F_Y$ and $\kappa^{-1}(\kappa(F_X) \cup \{y_1,\ldots, y_i, y'\})$: If $\kappa^{-1}(\kappa(F_X) \cup \{y_1,\ldots, y_i, y'\}) \not\sqsubseteq F_Y$,
    then as,
    \[
        \kappa^{-1}(\kappa(F_X) \cup \{y_1,\ldots, y_i, y'\}) \sqcap F_Y = \kappa^{-1}(\kappa(F_X) \cup \{y_1,\ldots, y_i\}),
    \]
    since $\kappa^{-1}(\kappa(F_X) \cup \{y_1,\ldots, y_i, y'\}) \succ \kappa^{-1}(\kappa(F_X) \cup \{y_1,\ldots, y_i\})$,
    it follows by the Forking Lemma that $y' \in \Gamma(F_Y)$.
    Therefore, by Lemma~\ref{lem:big-stick}.3 implies $y'\notin Y$, and
    as $\widetilde{\beta x}\subseteq X\subseteq Y$ this shows $y'\notin\widetilde{\beta x}$.
    Now assume $\kappa^{-1}(\kappa(F_X) \cup \{y_1,\ldots, y_i, y'\}) \sqsubseteq F_Y$.
    Observe, this forces $i < k$.
    Moreover, because $z\beta \in \glang/F_Y$ by inductive hypothesis, it is clear that $y' \notin \widetilde{\beta}$.
    By way of contradiction then, suppose $y' = x$.
    Because $y_1\ldots y_k$ is an arbitrary word in $F_Y/F_X$, we can assume $y_{i+1} = x$ without loss of generality. 
    From our inductive hypothesis,
    \[
        y_1\ldots y_i x y_{i+2} \ldots y_k z \beta \in \glang/F_X,
    \]
    so the union of the supports of $z\beta x$ and $y_1\ldots y_i x$ are contained in the support of some $(\glang/F_X)$-feasible word.
    Hence, as $z\beta \in \glang/F_Y$ implies $\widetilde{z\beta}\cap\{y_1,\ldots,y_i\} = \varnothing$, $x \in \glang/F_X$ follows by the local poset property (which is easily seen to be closed under contraction by examining Definition~\ref{def:local-poset}).
    But Lemma~\ref{lem:big-stick}.3 requires $x\notin\Gamma(F_X)$ as $x \in X$, a contradiction.
    So Eq.~\ref{eq:cant-use-x} holds.

    With that done define $F_0 \triangleq \kappa^{-1}(\kappa(F_Y) \cup \widetilde{z\beta})$ for notational convenience, and take some basic $\pi = c_1\ldots c_{r(\glang/F_0)} \in \glang/F_0$.
    We will examine a greedy optimization over $\glang/F_X$, and use Eq.~\ref{eq:cant-use-x} to show a contradiction whenever $z\beta x \not\in\glang/F_Y$.
    Note that it is folklore knowledge that contraction minors inherit the correctness properties of the greedy algorithm in our setting; we give a short proof in Appendix~\ref{app:greedy-minors}.
    So, let $\beta = b_1\ldots b_{k'}$ and weigh the letters of the alphabet according to,
    \begin{equation*}
        w(a) \triangleq \begin{cases}
            1 - (k - i + 1)\cdot\delta,\quad&\text{if }a = y_i,\\
            i\cdot\delta,\quad&\text{if }a=b_i,\\
            (k'+1)\cdot\delta,\quad&\text{if }a=x,\\
            1,\quad&\text{if }a=z,\\
            2 + i\cdot\delta,\quad&\text{if }a=c_i,\\
            \infty,\quad&\text{otherwise},
        \end{cases}
    \end{equation*}
    where $\delta > 0$ is a small real selected so that,
    \begin{equation}\label{eq:delta-bound}
        \max\{k,k' + 1\}\cdot\delta < \frac{1}{2}.
    \end{equation}
    Towards a contradiction, let's assume $z\beta x \notin\glang/F_Y$.
    One can identify that applying the greedy algorithm to $\glang/F_X$ will form the word $y_1\ldots y_k z\beta\pi$:
    Beginning at the empty word $\epsilon$, because the continuations of $\epsilon$ in $\glang/F_X$ equal $\Gamma(F_X)$, and $\widetilde{\beta x}\subseteq X$ makes $\widetilde{\beta x}\cap\Gamma(F_X)=\varnothing$ by Lemma~\ref{lem:big-stick}.3, the algorithm cannot select any of $\{b_1,\ldots,b_{k'}, x\}$. 
    Therefore, because our construction makes $w(y_1) < w(y_2) < \ldots < w(y_k) < w(z)$, and $w(z) < w(c_i)$ for all $c_i \in \widetilde{\pi}$, the greedy algorithm selects $y_1$ as its first letter.
    Then, due to Eq.~\ref{eq:cant-use-x}, similar logic shows that the greedy algorithm forms $y_1\ldots y_k$ over the next $k-1$ steps.
    For the $(k+1)^\text{th}$ step, see that $z$ is the continuation because $z\beta \in \glang/F_Y$ by our inductive hypothesis, and $y_1\ldots y_k\in F_Y/F_X$.
    Then, it follows (once again) from Eq.~\ref{eq:cant-use-x} that $z$ is the continuation of least weight.
    So, the greedy algorithm selects $z$.
    Following this, the greedy algorithm forms the word $y_1\ldots y_k z \beta$ over the next $k'$ steps, since each $b_i$ is a minimum weight continuation of $y_1\ldots y_k z b_1\ldots b_{i-1}$, as,
    $$w(b_i) < w(b_{i+1}) < \ldots < w(b_{k'}) < w(x) < w(c_1) < \ldots < w\left(c_{r(\glang/F_0)}\right).$$
    Here, take note that $y_1\ldots y_k z \beta \in F_0/F_X$ by $y_1\ldots y_k \in F_Y/F_X$ and our definition of $F_0$.
    Because $z\beta x \notin\glang/F_Y$ by assumption, it thus follows that $x\notin\Gamma(F_0)$.
    Furthermore,
    \begin{align*}
        \kappa^{-1}\left(\kappa(F_X) \cup \widetilde{z\beta x}\right) \succ \kappa^{-1}\left(\kappa(F_X) \cup \widetilde{z\beta}\right),&&\text{and},&&\kappa^{-1}\left(\kappa(F_X) \cup \widetilde{z\beta}\right) \sqsubseteq F_0,
    \end{align*}
    so $x\notin\Gamma(F_0)$ makes (by the contrapositive) of the Forking Lemma,
    \[
        \kappa^{-1}\left(\kappa(F_X) \cup \widetilde{z\beta x}\right) \sqsubseteq F_0.
    \]
    Hence, $x \in \kappa(F_0)$, meaning there is no $\xi\in\univ^*$ such that $y_1\ldots y_k z \beta \xi x$ can be feasible in $\glang/F_X$.
    Importantly, this forces the greedy algorithm to select continuations from $\widetilde{\pi}$ until it terminates with $y_1\ldots y_k z \beta \pi$, as we made the weight of every other remaining letter infinite.
    The weight of the resulting word is then,
    \begin{equation}\label{eq:greedy-value}
        w(\{y_1,\ldots, y_k, z\}\cup\widetilde{\beta}\cup\widetilde{\pi}) = 1 + \left(\sum_{i=1}^{k} 1 - (k-i + 1)\cdot \delta\right) + \left(\sum_{i=1}^{k'} i\cdot \delta\right) + w(\widetilde{\pi}).
    \end{equation}

    Now, recall $z\beta x \in \glang/F_X$. By applying the exchange property to $z\beta x$ and $y_1\ldots y_k z \beta \pi$ there is a word $\gamma\in\univ^*$ with $\widetilde{\gamma}\subseteq \{y_1,\ldots, y_k\}\cup\widetilde{\pi}$ which makes a basic word $z\beta x \gamma \in\glang/F_X$.
    Hence, $z\beta x \gamma$ is feasible for our optimization.
    But, we can bound the weight of this word as follows,
    \begin{equation*}
    \begin{split}
        w(\widetilde{\beta}\cup\widetilde{\gamma}&\cup\{x,z\}) \leq 1 + \left(\sum_{i=2}^{k} 1 - (k-i + 1)\cdot \delta\right) + \left(\sum_{i=1}^{k'+1} i\cdot \delta\right) + w(\widetilde{\pi}),\\
        &= (k+k' + 1)\cdot \delta +  \left(\sum_{i=1}^{k} 1 - (k-i + 1)\cdot \delta\right) + \left(\sum_{i=1}^{k'} i\cdot \delta\right) + w(\widetilde{\pi}).
    \end{split}
    \end{equation*}
    From Eq.~\ref{eq:delta-bound} it is apparent that $(k+k'+1)\cdot\delta <1$.
    Then by examining Eq.~\ref{eq:greedy-value} we find,
    $$w(\{y_1,\ldots, y_k, z\}\cup\widetilde{\beta}\cup\widetilde{\pi}) > w(\widetilde{\beta}\cup\widetilde{\gamma}\cup\{x,z\}),$$
    which means that the greedy algorithm fails to optimize $\alpha\mapsto \sum_{a\in\widetilde{\alpha}} w(a)$ over $\glang/F_X$, a contradiction.
    Hence $z\alpha' = z\beta x\in\glang/F_Y$, which makes $z\alpha\in\glang/F_Y$ by the principle of induction.
\end{proof}

\begin{remark}
    The observant reader noticed that the inequality in Lemma~\ref{lem:key-lemma} looks reminiscent of a reversal of the submodular law.
    With some more observations one can observe $\left(\greedyrep/Y\right)(z)\neq 0$ implies $F_X = F_{X+z} \sqcap F_Y$ and $F_{Y+z} = F_{X+z}\sqcup F_Y$.
    Hence, $r(F_{Y+z}) - r(F_Y) \leq r(F_{X+z}) - r(F_X)$ as well, so $F_{X+z}$ and $F_Y$ make a modular pair.
    From Proposition~\ref{prop:greedy-aligned-rep}, one sees this specializes further in the case of matroids as we'd have $F_{X+z}\succ F_X$ and $F_Y\prec F_{Y+z}$.
\end{remark}

Now, we have the ingredients necessary to completely verify Main Result~\ref{mr:lp-greedy=>pm-greedoid}.

\begin{proof}[Proof of Main Result~\ref{mr:lp-greedy=>pm-greedoid}]
    First we verify that $\greedyrep$ is a polymatroid rank function.
    Well, it is clear that $\greedyrep(\varnothing) = r[\epsilon]$, so $\greedyrep$ is normalized.
    Furthermore, monotonicity was proven in Proposition~\ref{prop:gr-monotonicity}.
    Now, fix $X,Y\subseteq \univ$ and $z \in\univ$.
    For submodularity we will verify the law of diminishing returns, i.e. $X\subseteq Y$ implies $\left(\greedyrep/X\right)(z)\geq \left(\greedyrep/Y\right)(z)$.
    So, assume $X\subseteq Y$.
    By monotonicity, $\left(\greedyrep/X\right)(z) \geq 0$.
    And so the desired inequality is clear if $\left(\greedyrep/Y\right)(z) = 0$.
    In light of this, suppose $\left(\greedyrep/Y\right)(z) \neq 0$.
    As outlined before, one applies Lemma~\ref{lem:key-lemma} to Eq.~\ref{eq:key1} to obtain Eq.~\ref{eq:key2}.
    Recalling that Eq.~\ref{eq:key1} is equivalent to $\lfloor\psi_{F_{X+z}}(X+z)\rfloor -\lfloor \psi_{F_X}(X)\rfloor = \lfloor\psi_{F_{Y+z}}(Y+z)\rfloor -\lfloor \psi_{F_Y}(Y)\rfloor$, this equality and Eq.~\ref{eq:key2} imply,
    \begin{multline*}
            \left(\greedyrep/X\right)(z) = \lfloor\psi_{F_{X+z}}(X+z)\rfloor -\lfloor \psi_{F_X}(X)\rfloor + \frac{b_{F_{X+z}}(X+z) -  b_{F_X}(X)}{r(\glang)+1},\\
                                         \geq \lfloor\psi_{F_{Y+z}}(Y+z)\rfloor -\lfloor \psi_{F_Y}(Y)\rfloor + \frac{b_{F_{Y+z}}(Y+z) -  b_{F_Y}(Y)}{r(\glang)+1} = \left(\greedyrep/Y\right)(z),
    \end{multline*}
    hence the law of diminishing returns.
    Thus, $\greedyrep$ is a polymatroid rank function.
    And so, by Proposition~\ref{prop:greedy-aligned-rep} we see $\greedyrep$ is a representation.
    What remains is verifying alignment.
    We inspect the composition $\varphi = \sigma_{\greedyrep}\circ\kappa$.
    The inclusion axiom should be clear.
    For agreement, select $\alpha \in \glang$ and $x \in \kappa[\alpha]\setminus\widetilde{\alpha}$.
    By the contrapositive of Proposition~\ref{prop:strong-opt}, there exists no $\beta \in \glang/\alpha$ such that $x \widetilde{\beta}$.
    Suppose $\alpha\beta$ is feasible and basic.
    Combining our previous observation with optimism shows the existence of a strict $\alpha$-prefix $\alpha'$ with $x \in \Gamma[\alpha']$.
    Therefore $(\greedyrep/\widetilde{\alpha})(x) < 1$ by $\alpha x \notin \glang$ and the law of diminishing returns.
    So, $(\greedyrep/\widetilde{\alpha})(x) = 0$ by Proposition~\ref{prop:greedy-aligned-rep}.
    Since our choice of $x \in \kappa[\alpha]\setminus\widetilde{\alpha}$ was arbitrary it follows $(\greedyrep/\widetilde{\alpha})(\kappa[\alpha]) = 0$, and so $\greedyrep(\kappa[\alpha]) = \greedyrep(\widetilde{\alpha}) = |\alpha|$ by submodularity.
    This shows that $\varphi$ satisfies agreement.
    Finally, select flats $F,F'\in\glang/\mathord{\sim}$ such that $F\prec F'$.
    We note that for there to exist a $\greedyrep$-closed set between $\varphi(F)$ and $\varphi(F')$, there would have to exist some letter $x \in \univ$ with $0 < (\greedyrep/\varphi(F))(x) < 1$.
    However, agreement shows $\sigma_{\greedyrep}(\widetilde{\alpha}) = \varphi(F)$ for all $\alpha \in F$.
    So, this is equivalent to $0 < (\greedyrep/\widetilde{\alpha})(x) < 1$ for some $\alpha \in F$, which is impossible by Proposition~\ref{prop:greedy-aligned-rep}.
\end{proof}

\section{Galois Connections}\label{sec:galois}
We now prove Main Result~\ref{mr:gc}.
Besides extending the approach to polymatroid greedoids, the apparent order-theoretic duality emphasizes the fundamental nature of the relationship between a greedoid and its aligned representations.

We equate the existence of aligned representations to $\kappa$ being a cover-preserving lower adjoint.
As adjoints uniquely determine each other, explicit computation shows the inverse kernel closure operator is the upper adjoint.
See Figure~\ref{fig:galois} for an example of this interaction.
As an application, in Section~\ref{subsec:lattice-embeddings} we will show that the flat supports being closed under intersection provides a necessary and sufficient condition of possessing a certain maximum representation.

\begin{figure}[t!]
    \centering
    \scalebox{.8}{
    \input{tikz/galois.tikz}
    }
    \caption{The Galois connection between the flats of the undirected branching greedoid in Figure~\ref{fig:ubg} and the closed sets of the representation $\rho$ in Example~\ref{ex:ubg-rep}. 
        We see $\L(\glang)$ is isomorphic to a sublattice of $\L(\rho)$.
        But this does not hold generally, see Main Result~\ref{mr:strong-pm-greedoid}.
    }\label{fig:galois}
\end{figure}

\begin{proof}[Proof of Main Result~\ref{mr:gc}]
    First, we show alignment implies the above three statements.
    For 1., select any $y\notin\kappa[\alpha]$.
    By Proposition~\ref{prop:strong-opt}, there exists $\beta\in\glang/\alpha$ such that $\alpha\beta y \in \glang$.
    It follows that $(\rho/\widetilde{\alpha\beta})(y) = 1$, so $(\rho/\widetilde{\alpha})(y) \geq 1$ by the law of diminishing returns.
    Equivalently, $y \notin \sigma_\rho(\widetilde{\alpha})$.
    Because $y$ was an arbitrary letter beyond the flat support, this shows $\sigma_\rho(\widetilde{\alpha})\subseteq \kappa[\alpha]$ by contraposition.
    Continuing on, let $\varphi:\mathscr{L}(\glang)\to\mathscr{L}(\rho)$ be the mapping witnessing alignment.
    By agreement and inclusion, $\varphi[\alpha] \subseteq \sigma_\rho(\widetilde{\alpha})$.
    But inclusion also makes $\beta\sim\alpha$ imply $\widetilde{\beta}\subseteq \varphi[\alpha]$, so
    $\kappa[\alpha]\subseteq \varphi[\alpha]$ as well.
    This shows,
    \begin{equation}\label{eq:keqs}
        \sigma_\rho(\widetilde{\alpha}) \subseteq \kappa[\alpha] \subseteq \varphi[\alpha] \subseteq\sigma_\rho(\widetilde{\alpha}),
    \end{equation}
    and so $\kappa[\alpha]$ is $\rho$-closed, hence 1. follows.

    Observe, Eq.~\ref{eq:keqs} shows that $\kappa = \varphi$, and so $\kappa$ is cover-preserving by the alignment axioms.
    Another consequence of Eq.~\ref{eq:keqs} is we may restrict the codomain of $\kappa$ to $\L(\rho)$ without problem.
    Conveniently, because $\kappa$ defines an isomorphism between $\L(\glang)$ and the flat supports ordered by containment, it follows that $\kappa$ is order-preserving and order-reflecting onto $\L(\rho)$.
    This makes $\kappa$ an order-embedding between $\L(\glang)$ and $\L(\rho)$.
    Recall $\L(\glang)$ is semimodular.
    Lemma 1 of \cite{wild1993cover} states that every cover-preserving order-embedding originating from a semimodular lattice is join-preserving.
    Therefore, $\kappa$ is a lower adjoint by Lemma~\ref{lem:gc-join}.
    As adjoints uniquely determine each other, using Eq.~\ref{eq:adjoints} we see,
    \[
        \varphi^*(X) = \bigsqcup\{F\in\glang/\mathord{\sim}\mid \kappa(F) \subseteq X\} = \kappa^{-1}\left(\bigcup\{\widetilde{\alpha}\mid \alpha \in \glang\cap X!\}\right) = \kappa^{-1}(X).
    \]
    Hence, $\kappa:\L(\glang)\rightleftarrows\L(\rho):\kappa^{-1}$ is a Galois connection.
    This gives 3., so alignment implies (1-3).

    Now, for the other direction assume (1-3) are true; we prove that $\kappa$ is a mapping $\L(\glang)\to\L(\rho)$ satisfying the alignment axioms.
    1. shows that this task is well-defined, while 2. shows that $\kappa$ is cover-preserving.
    Inclusion also follows by the definition of $\kappa$ making $\widetilde{\alpha}\subseteq\kappa[\alpha]$.
    Note, this implies $\sigma_\rho(\widetilde{\alpha}) \subseteq \kappa[\alpha]$.
    Then by the definition of a Galois connection, $\sigma_\rho(\widetilde{\alpha}))\subseteq \kappa[\alpha]$ implies $\kappa^{-1}(\sigma_\rho(\widetilde{\alpha}))\sqsubseteq [\alpha]$.
    But, $[\alpha] = \kappa^{-1}(\widetilde{\alpha})$ makes $[\alpha]\sqsubseteq \kappa^{-1}(\sigma_\rho(\widetilde{\alpha}))$ as well, so we have $[\alpha] = \kappa^{-1}(\sigma_{\rho}(\widetilde{\alpha}))$.
    Continuing on, recall that $\kappa\circ\kappa^{-1}$ is an interior operator.
    As it is deflationary, we have $\sigma_\rho(\widetilde{\alpha})\supseteq \kappa(\kappa^{-1}(\sigma_\rho(\widetilde{\alpha}))$.
    Then, substituting $[\alpha] = \kappa^{-1}(\sigma_{\rho}(\widetilde{\alpha}))$ shows $\sigma_\rho(\widetilde{\alpha})\supseteq  \kappa[\alpha]$.
    Combining this with our initial observation $\sigma_\rho(\widetilde{\alpha})\subseteq\kappa[\alpha]$, we see $\sigma_\rho(\widetilde{\alpha}) = \kappa[\alpha]$.
    And so we have agreement.
\end{proof}

\subsection{Strong Polymatroid Greedoids and the Greatest Representation}\label{subsec:lattice-embeddings}

Define the function $\accentset{\vee}{\rho}:2^\univ \to \mathbb{Z}_+$ as one assigning a value to each subset $X\subseteq\univ$ equal to the minimum rank among flats containing $X$ in its flat support,
\[
    \accentset{\vee}{\rho}(X) \triangleq \min\{r(F)\mid X\subseteq\kappa(F)\}.
\]
Suppose that $\accentset{\vee}{\rho}$ is a representation of a greedoid $\glang$.
It is not difficult to argue that $\accentset{\vee}{\rho}$ is then the unique maximum aligned representation of $\glang$.
In light of this, we call $\accentset{\vee}{\rho}$ the \emph{greatest representation}.
\begin{proposition}\label{prop:unique-max}
    Let $\glang$ be a polymatroid greedoid.
    If the greatest representation $\accentset{\vee}{\rho}$ is a polymatroid rank function representing $\glang$, then $\accentset{\vee}{\rho}$ is the unique maximum aligned representation of $\glang$.
\end{proposition}
\begin{proof}
    Observe, if $\accentset{\vee}{\rho}$ is both a polymatroid rank function and a representation, then because $\accentset{\vee}{\rho}$ is integral Proposition~\ref{prop:int->aligned} would show $\accentset{\vee}{\rho}$ is aligned.
    With that settled, let $\rho$ be any aligned representation of $\glang$, and select any subset $X\subseteq \univ$.
    Fix a flat $F\in\glang/\mathord{\sim}$ with $X \subseteq \kappa(F)$ and $\accentset{\vee}{\rho}(X) = r(F)$.
    Then, $X\subseteq\kappa(F)$ implies $\rho(X) \leq \rho(\kappa(F))$ since $\rho$ is a polymatroid rank function.
    But $\rho(\kappa(F)) = r(F)$ because $\rho$ is aligned.
    Therefore,
    \[
        \rho(X) \leq \rho(\kappa(F)) = r(F) = \accentset{\vee}{\rho}(X),
    \]
    and we conclude the result.
\end{proof}

We will see that possessing the greatest representation implies nice greedoid structure.
For this reason, we name the class of greedoids with the greatest representation \emph{strong polymatroid greedoids}.

In what follows we will derive properties of strong polymatroid greedoids, and describe them combinatorially.
To begin the latter, first consider a greedoid whose flat supports are closed under intersection.
If $\kappa(F)\cap\kappa(F')$ is a flat support then $\kappa^{-1}\left(\kappa(F)\cap \kappa(F')\right)$ must be lesser than both of $F$ and $F'$ by the isomorphism between the flat supports and the flats of an interval greedoid.
This shows $\kappa(F) \cap \kappa(F') \subseteq \kappa(F\sqcap F')$.
Furthermore, $\kappa(F\sqcap F') \subseteq \kappa(F)\cap \kappa(F')$ is always true because $\kappa:\L(\glang)\to\left(2^\univ,\subseteq\right)$ is order-preserving.
Therefore,
\begin{equation}\label{eq:int-closure}
    \kappa(F)\cap\kappa(F') = \kappa(F\sqcap F'),\quad\forall F,F'\in\glang/\mathord{\sim}.
\end{equation}
So, for every subset $X\subseteq \univ$ there is a unique minimal flat containing $X$ in its flat support, definitionally given by the meet of flats containing $X$ in their flat supports.
The rank is the value given to $X$ by $\accentset{\vee}{\rho}$, i.e.,
\[
    \accentset{\vee}{\rho}(X) = r\left(\bigsqcap \{F\in\glang/\mathord{\sim}\mid X\subseteq \kappa(F)\}\right).
\]
We use this argue that the greatest representation is a polymatroid rank function representing any optimistic interval greedoid whose flat supports are closed under intersection.

\begin{proposition}\label{prop:gr-is-rep}
    Let $\glang$ be a greedoid with optimism and the interval property.
    If the flat supports are closed under intersection, then the greatest representation $\accentset{\vee}{\rho}$ is a polymatroid rank function representing $\glang$.
\end{proposition}
\begin{proof}[Proof of Proposition~\ref{prop:gr-is-rep}]
    Fix subsets $X, Y\subseteq \univ$ and let,
    \begin{align*}
        M_X \triangleq \bigsqcap\{F\in\glang/\mathord{\sim}\mid X\subseteq\kappa(F)\}.
    \end{align*}
    Define $M_Y$ similarly.
    Our construction makes $X\cup Y \subseteq \kappa(M_X\sqcup M_Y)$, and $X\cap Y \in \kappa(M_X\sqcap M_Y)$ by assuming the flat supports are closed under intersection.
    By semimodularity,
    \[
        \accentset{\vee}{\rho}(X) + \accentset{\vee}{\rho}(Y) = r(M_X) + r(M_Y) \geq r(M_X\sqcap M_Y) + r(M_X \sqcup M_Y) \geq \accentset{\vee}{\rho}(X\cap Y) + \accentset{\vee}{\rho}(X\cup Y).
    \]
    Thus, $\accentset{\vee}{\rho}$ is submodular.
    Furthermore, $\accentset{\vee}{\rho}$ is clearly normalized, and monotonicity follows from the fact that $X\subseteq Y$ would imply $M_X \sqsubseteq M_Y$ in our construction.
    So $\accentset{\vee}{\rho}$ is a polymatroid rank function.

    Now, to show it is a representation we must first verify,
    \begin{equation}\label{eq:ufntu}
        M_{\widetilde{\alpha}} = [\alpha],\quad\forall\alpha\in\glang,
    \end{equation}
    where $M_{\widetilde{\alpha}}$ is defined as before.
    Fix $\alpha \in \glang$, and see that $\widetilde{\alpha} \subseteq \kappa[\alpha]$ implies $M_{\widetilde{\alpha}}\sqsubseteq [\alpha]$.
    Because the flat supports are closed under intersection, we see $\widetilde{\alpha}\subseteq \kappa(M_{\widetilde{\alpha}})$.
    So, there is no $z \in \widetilde{\alpha}$ such that $z\in\Gamma(M_{\widetilde{\alpha}})$.
    To maintain the exchange property the only possibility is $M_{\widetilde{\alpha}}= [\alpha]$, and so Eq.~\ref{eq:ufntu} holds.
    Continuing on, fix $\alpha \in \glang$ and $z\in\univ$.
    First assume $z\in\Gamma[\alpha]$.
    Then $[\alpha]\prec[\alpha z]$ and Eq.~\ref{eq:ufntu} show the desired $\left(\accentset{\vee}{\rho}/\widetilde{\alpha}\right)(z) = 1$.
    Now assume $z \notin\Gamma[\alpha]$.
    If $z \in \kappa[\alpha]$, then once again with Eq.~\ref{eq:ufntu} it follows that the minimum flat containing $\widetilde{\alpha}+z$ in its flat support is $[\alpha]$ itself.
    Therefore, $z \in \kappa[\alpha]$ implies $\left(\rho/\widetilde{\alpha}\right)(z) = 0$.
    In the other case, where $z \notin\kappa[\alpha]$, there cannot exist any flat covering $[\alpha]$ containing $\widetilde{\alpha} + z$ in its flat support.
    This is because if such a flat did exist then $z$ would have to be a continuation by Proposition~\ref{prop:strong-opt}.
    Our insight here applies to $M_{\widetilde{\alpha}+z}$ in particular.
    Observe that $X \mapsto M_X$ is order-preserving between $\left(2^\univ,\subseteq\right)$ and $\L(\glang)$, and so Eq.~\ref{eq:ufntu} implies $M_{\widetilde{\alpha} + z}$ lies above $[\alpha]$.
    And so, because $M_{\widetilde{\alpha} + z}$ cannot cover $[\alpha]$ and $z\notin\kappa[\alpha]$ implies $M_{\widetilde{\alpha} + z}\neq[\alpha]$, the rank of $M_{\widetilde{\alpha}+z}$ is at least 2 greater than that of $[\alpha]$.
    Thus, we find $\left(\accentset{\vee}{\rho}/\widetilde{\alpha}\right)(z) \geq 2$.
\end{proof}

The flat supports being a Moore family is also shown as necessary for possessing the greatest representation.
In particular, we use the closure theory of $\kappa:\L(\glang)\rightleftarrows\L(\rho):\kappa^{-1}$ to show $\L(\glang)\cong\L(\accentset{\vee}{\rho})$.
Then, since $\L(\glang)\cong(2^\univ,\subseteq)$, this isomorphism relates the meet of $\L(\accentset{\vee}{\rho})$ (i.e. intersection) to that of the flat supports.
Using the adjoints, we also find greedoids with the greatest representation make lattice-embeddings for \emph{all} aligned representations $\rho$.
Such polymatroid greedoids are especially well-behaved, as their flats are always isomorphic to a sublattice of $\L(\rho)$.
This is Main Result~\ref{mr:strong-pm-greedoid}.
\begin{proof}[Proof of Main Result~\ref{mr:strong-pm-greedoid}]
    Firstly, see that 1.$\iff$5. by definition.
    We proceed by showing $2.\iff 5.$ and $5. \implies 3.\implies 4. \implies 5.$
    First we verify $5.\implies 2.$.
    Proposition~\ref{prop:gr-is-rep} shows $\accentset{\vee}{\rho}$ is an integral representation of $\glang$, and so Main Result~\ref{mr:gc} makes $\sigma_{\accentset{\vee}{\rho}}\circ\kappa = \kappa$ and $\kappa$ a lower adjoint.
    By assuming 5. we find $\kappa$ is meet-preserving.
    Furthermore, by Lemma~\ref{lem:gc-join} $\kappa$ is also join-preserving, hence $\sigma_{\accentset{\vee}{\rho}}\circ\kappa = \kappa$ is a lattice-embedding.
    Now for $2. \implies 5.$, because $\glang$ is a polymatroid greedoid we can once again let $\sigma_{\accentset{\vee}{\rho}}\circ\kappa = \kappa$ by Main Result~\ref{mr:gc}.
    So $\kappa$ is a lattice-embedding.
    Then, because the meet operation on the closed sets of a polymatroid is intersection, this makes $\kappa(F\sqcap F') = \kappa(F)\cap\kappa(F')$.

    With that done, we see $5.\implies 3.$ from Proposition~\ref{prop:gr-is-rep}. 
    For $3. \implies 4.$, first note that because 3. makes $\glang$ possess an aligned representation, we may use the Galois connection $\kappa:\L(\glang)\rightleftarrows\L(\accentset{\vee}{\rho}):\kappa^{-1}$ by Main Result~\ref{mr:gc}.
    We show 4. by proving that the interior operator $\kappa\circ\kappa^{-1}$ equals the identity mapping over $\L(\accentset{\vee}{\rho})$.
    This is sufficient because if both $\kappa\circ\kappa^{-1}$ and $\kappa^{-1}\circ\kappa$ equal the identity mappings on their codomains then $\kappa$ and $\kappa^{-1}$ are inverses.
    Then $\kappa:\L(\glang)\to\L(\accentset{\vee}{\rho})$ would define an order-preserving bijection, i.e. an isomorphism.
    Finally, we already know $\kappa^{-1}\circ\kappa$ is the identity mapping because $\kappa$ is injective (due to the correspondence between flats and the flat supports), so we need only examine $\kappa\circ\kappa^{-1}$.
    Fix some $X\subseteq \univ$.
    By inspecting the definition of $\accentset{\vee}{\rho}$ we see that,
    \[
        \sigma_{\accentset{\vee}{\rho}}(X) = \bigcup\left\{\kappa(F)\mid X\subseteq\kappa(F)\text{ and }(\forall F'\in\glang/\mathord{\sim})\;X\subseteq\kappa(F')\implies r(F) \leq r(F')\right\}.
    \]
    Since $\sigma_{\accentset{\vee}{\rho}}(X)$ is a union of flat supports, it follows that it is a partial alphabet.
    Hence, $\sigma_{\accentset{\vee}{\rho}}(X) \subseteq \kappa(\sigma_{\accentset{\vee}{\rho}}(X))$, which is equivalent to
    $\sigma_{\accentset{\vee}{\rho}}(X)\subseteq \kappa(\kappa^{-1}(\sigma_{\accentset{\vee}{\rho}}(X)))$.
    However, because the interior operator is deflationary we also have $\kappa(\kappa^{-1}(\sigma_{\accentset{\vee}{\rho}}(X))) \subseteq \sigma_{\accentset{\vee}{\rho}}(X)$, and so $\sigma_{\accentset{\vee}{\rho}}(X) = \kappa(\kappa^{-1}(\sigma_{\accentset{\vee}{\rho}}(X)))$.
    Thus 4.

    The final component is proving $(4.)\implies (5.)$.
    Well, fix $\alpha \in \glang$ and see that $[\alpha]$ is the unique flat of minimum rank containing the feasible support $\widetilde{\alpha}$ in its flat support.
    By definition of $\accentset{\vee}{\rho}$, this implies $\sigma_{\accentset{\vee}{\rho}}(\widetilde{\alpha}) = \kappa[\alpha]$.
    Therefore, all flat supports are $\accentset{\vee}{\rho}$-closed sets.
    Recall $\L(\glang)\cong\left(\kappa\left(2^\univ\right),\subseteq\right)$ due to the interval property, so by transitivity of isomorphism we have $\left(\kappa\left(2^\univ\right),\subseteq\right)\cong\L(\accentset{\vee}{\rho})$.
    This makes the flat supports and $\accentset{\vee}{\rho}$-closed sets equivalent (as both are defined as subposets of $(2^\univ,\subseteq)$), so the flat supports must form a Moore family.
    Hence, the flat supports are closed under intersection.
\end{proof}

Via results from Section~\ref{sec:pm-greedoid}, a straightforward consequence is that a greedoid is a strong polymatroid greedoid if and only if its flat supports are a Moore family.

\begin{corollary}\label{cor:strong-pm}
    Let $\glang$ be a greedoid.
    Then, $\glang$ is a strong polymatroid greedoid if and only if it is an optimistic interval greedoid with flat supports forming a Moore family.
\end{corollary}
\begin{proof}
    Suppose $\glang$ is a strong polymatroid greedoid.
    Theorems of \cite{korte1985polymatroid} and Theorem~\ref{thm:pm->optimism} ensure any greedoid with an aligned representation is an optimistic interval greedoid, while the flat supports are a Moore family by definition.
    The other direction follows by Main Result~\ref{mr:strong-pm-greedoid}.
\end{proof}

\begin{figure}[t!]
    \centering
    \scalebox{.8}{\input{tikz/cycle-graph.tikz}}
    \caption{
        Consider an undirected branching greedoid formed from the above cycle graph, with root node $s$.
        We can see that there are exactly two minimal flats containing $\{x\}$ in their flat supports: $[a_1\ldots a_i x]$ and $[b_1\ldots b_j x]$.
        Assume $\max\{i,j\} > 1$ so that $[a_1\ldots a_i x] \neq [b_1\ldots b_j x]$.
        Then, the meet follows $[a_1\ldots a_i x] \sqcap [b_1\ldots b_j x] = [\epsilon]$, but $\kappa[\epsilon]$ clearly does not contain $x$.
    }\label{fig:flat support-not-closed}
\end{figure}
Because our description is necessary and sufficient, not every polymatroid greedoid can possess the greatest representation (flat supports cannot always form a Moore family).
For example, though the undirected branching greedoid given in Figure~\ref{fig:ubg} satisfies this condition, this is not true in general.
See Figure~\ref{fig:flat support-not-closed} for an example using a cycle graph.

However, there are many interesting examples which do possess the greatest representation.
One immediate example is a matroid, as its flat supports correspond to its rank closed sets (which are a Moore family).
Another example is \emph{poset antimatroids}, that is a greedoid whose feasible words are the linear extensions of a poset, since one can identify that its flat supports are the ideals of a poset (which are easily seen to be closed under intersection). 
A generalization of both matroids and poset antimatroids are \emph{distributive supermatroids}.
\begin{definition}[Distributive Supermatroid]
    A greedoid $\glang$ is a \emph{distributive supermatroid} if and only if there exists a poset $\mathscr{P}$ such that,
    \begin{enumerate}
        \item $\alpha \in \glang$ implies $\widetilde{\alpha}$ is an ideal of $\mathscr{P}$,
        \item $\alpha \in \glang$, $X$ is an ideal of $\mathscr{P}$, and $X\subseteq \widetilde{\alpha}$ imply $X!\cap\glang \neq \varnothing$.
    \end{enumerate}
\end{definition}

Distributive supermatroids are notable because they make a generalization of Edmond's matroid intersection theorem \cite{tardos1990intersection, edmonds1979matroid}.
Here, we show \emph{optimistic} distributive supermatroids are a notable class possessing the greatest representation, by using the description in Corollary~\ref{cor:strong-pm}.

\begin{proposition}
    Let $\glang$ be a distributive supermatroid, and suppose $\glang$ is optimistic.
    Then, $\glang$ possesses the greatest representation.
\end{proposition}
\begin{proof}
    Distributive supermatroids possess the interval property \cite{korte2012greedoids}, and we are including optimism in our assumptions.
    Therefore, by Corollary~\ref{cor:strong-pm} we need only verify that the flat supports of such greedoids are closed under intersection.
    Let $F\in\glang/\mathord{\sim}$.
    Then, by the interval property it follows that $\kappa(F)$ is a union of ideals on some poset.
    However, the union of ideals is once again an ideal.
    Now pick another flat $F'\in\glang/\mathord{\sim}$.
    Because ideals are closed under intersection, it follows that $\kappa(F)\cap\kappa(F')$ is also an ideal.
    However, by the distributive law,
    \begin{equation}\label{eq:asdfsadfjjj}
        \kappa(F)\cap\kappa(F') = \left(\bigcup\{\widetilde{\alpha} \mid \alpha \in F\}\right) \cap \left(\bigcup\{\widetilde{\beta}\mid \beta \in F'\}\right) = \bigcup\{\widetilde{\alpha}\cap\widetilde{\beta}\mid (\alpha,\beta)\in F\times F'\}.
    \end{equation}
    Observe, $(\alpha,\beta)\in F\times F'$ implies $\widetilde{\alpha}\cap\widetilde{\beta}$ is an ideal which is a subset of a feasible support (either of $\widetilde{\alpha}$ or $\widetilde{\beta}$ suffice), and so this intersection can be permuted into a feasible word.
    Therefore, Eq.~\ref{eq:asdfsadfjjj} implies $\kappa(F)\cap\kappa(F')$ is a union of feasible supports, i.e. a partial alphabet.
    This makes,
    \[
        \kappa(F)\cap\kappa(F') \subseteq \kappa(\kappa(F)\cap\kappa(F')) = \kappa(F\sqcap F'). 
    \]
    However, because $\kappa:\L(\glang)\to(2^\univ,\subseteq)$ is order-preserving, we also have $\kappa(F \sqcap F') \subseteq \kappa(F)\cap\kappa(F')$.
    Hence, $\kappa(F)\cap\kappa(F') = \kappa(F\sqcap F')$.
\end{proof}

\section*{Acknowledgements}
The authors would like to thank a set of anonymous referees for their helpful comments and encouragement, as well as identification of an error in the original proof of Lemma~\ref{lem:forking}.
\printbibliography
\newpage

\appendix

\section{Optimism and Strong Exchange}\label{app:optimism}

The purpose of this section is to show that every greedoid with the strong-exchange property, greedoids for which the greedy algorithm always solves any linear optimization over its basic words, possess optimism.
Our technique is to use a similar trick as that employed in the proof of Lemma~\ref{lem:key-lemma} to show an instance where the greedy algorithm fails in absence of optimism.

\begin{proposition}\label{prop:strong-exchange=>optimism}
    Let $\glang$ be a greedoid with the strong-exchange property.
    Then, $\glang$ is optimistic.
\end{proposition}
\begin{proof}
    Select a basic word $\alpha \in\glang$.
    To show $\glang$ is optimistic, it is sufficient to prove that for every $x \in \univ$ which is not a loop, there is an $\alpha$-prefix $\alpha'$ with $x \in \Gamma[\alpha']$.
    This statement is trivially true for $x \in \widetilde{\alpha}$, so we examine the case of $x \notin \widetilde{\alpha}$.
    Proceed by way of contradiction: Suppose there exists no $\alpha$-prefix containing $x$ in its set of feasible continuations.
    Then, let $\alpha = y_1\ldots y_r$ and construct a weighting $w:\univ\to\mathbb{R}$ via,
    \begin{equation*}
        w(z) \triangleq \begin{cases}
            \frac{i}{r(\glang)+1} - 1,\quad&\text{if } z = y_i,\\
            -r(\glang),\quad&\text{if }z = x,\\
            0,\quad&\text{otherwise}.
        \end{cases}
    \end{equation*}
    We claim that in the $i^\text{th}$ iteration, the greedy algorithm selects the letter $y_i$.
    This is true in the first step because $\epsilon$ is an $\alpha$-prefix, meaning that $x\notin\Gamma[\epsilon]$ by assumption which leaves $y_1$ as the letter of least weight which is initially a continuation.
    The same principle holds for the $i^\text{th}$ step, as if the greedy algorithm has already formed the word $y_1\ldots y_{i-1}$ then $x$ is not a feasible continuation by our assumptions, and $y_i$ is thus the continuation of least weight.
    So, it is true that the greedy algorithm forms the word $\alpha$ under this weighting by the principle of induction.
    However, because $x$ is not a loop there exists some basic word $\beta \in \glang$ containing $x$ in its support.
    Observe,
    \[
        \sum_{z \in \widetilde{\beta}}w(z) \leq w(x) < \sum_{z \in \widetilde{\alpha}} w(z).
    \]
    Hence, greedy fails to output the optimal basic word, contradicting strong-exchange.
\end{proof}
\begin{remark}\label{rem:optimism=/>strong-exchange}
    The existence of polymatroid greedoids, and thus optimistic greedoids, without the strong-exchange property is already known.
    One can construct such a greedoid using maximal ordered geometries (see Ch. 7 and 8 of \cite{korte2012greedoids}) for example.
\end{remark}

\section{An Optimistic Local Poset Greedoid with no Aligned Polymatroid Representation}\label{app:local-augmentation}

In this section we show the existence of an optimistic local poset greedoid which cannot possess an aligned representation.
We adapt an example proposed in \cite{korte1988intersection} to refute a conjecture that (integral) polymatroid greedoids are exactly local poset greedoids possessing a \emph{local augmentation} axiom.
To obtain an optimistic greedoid we do require a small change, which is pointed out in Remark~\ref{rem:change}.
Because of this and our slight difference in purpose, we present an analysis which does not utilize intermediate lemmas given in \cite{korte1988intersection}.

\begin{figure}
\begin{tikzpicture}[rotate=180,scale=1.5]
    \node (A) at (0, 1.155) [draw, circle] {};        
    \node (B) at (-1, -0.577) [draw, circle] {};      
    \node (C) at (1, -0.577) [draw, circle] {};       
    
    \draw (A) -- node[midway, right]{$\overline{q}$} (B);
    \draw (B) -- node[midway, above]{$\overline{p}$} (C);
    \draw (C) -- node[midway, left]{$\overline{r}$} (A); 

    \node (D1) at (-1, -2.577) [draw, circle] {};
    \node (D2) at (1, -2.577) [draw, circle] {};

    \draw (B) -- node[midway, right]{$\overbar{a_3}$} (D1);
    \draw (D1) -- node[midway, above]{$\overbar{a_2}$} (D2);
    \draw (D2) -- node[midway, left]{$\overbar{a_1}$} (C);

    \node (E1) at (2.732, 0.423) [draw, circle] {};    
    \node (E2) at (1.732, 2.155) [draw,circle] {};    

    \draw (C) -- node[midway, above]{$\overbar{c_3}$} (E1);
    \draw (E1) -- node[midway, left]{$\overbar{c_2}$} (E2);
    \draw (E2) -- node[midway, below]{$\overbar{c_1}$} (A);

    \node (F1) at (-2.732, 0.423) [draw, circle] {};
    \node (F2) at (-1.732, 2.155) [draw, circle] {};

    \draw (B) -- node[midway, above]{$\overbar{b_1}$} (F1);
    \draw (F1) -- node[midway, right]{$\overbar{b_2}$} (F2);
    \draw (F2) -- node[midway, below]{$\overbar{b_3}$} (A);
\end{tikzpicture}
\caption{A graph $(V, E)$ used to construct a certain greedoid with no polymatroid representation in \cite{korte1988intersection}.}\label{fig:counterexample}
\end{figure}

Consider the graph $(V, E)$ as defined in Figure~\ref{fig:counterexample}.
Define the function $f: E \to \mathbb{Z}_+$ so that $f(X)$ is equal to the number of edges in a spanning tree of the subgraph induced by restricting to the edges in $X$ (i.e. $f$ is just the rank function of the \emph{graphic matroid} defined over $(V,E)$).
Next, truncate $f$ so that its maximum value is 6, which is to say define another function $f':E\to\mathbb{Z}_+$ such that $f'(X) \triangleq \min\{f(X), 6\}$.
Now, we will define the alphabet of our greedoid.
Specifically, let,
\[
    a_i = \{\overbar{a_i}\}, b_i = \{\overbar{b_i}\},\text{ and } c_i = \{\overbar{c_i}\},\quad\forall i \in\{1,2,3\},
\]
as well as $y_1 = y_2 = y_3 = \{\overline{p},\overline{q},\overline{r}\}$.
Then, for convenience also let $A = \{a_1, a_2, a_3\}, B = \{b_1, b_2, b_3\}, C = \{c_1, c_2, c_3\}$, and $Y = \{y_1, y_2, y_3\}$ so that we can define our greedoid's alphabet as $\univ = A\cup B\cup C\cup Y$.
In a similar fashion to the way we define representations of polymatroid greedoids (see Eq.~\ref{eq:rep}), define a language $\glang_0$ over $\univ$ as follows:
\[
    \glang_0 \triangleq \left\{x_1\ldots x_i \in \univ^*\Biggm\vert \left(f'/\bigcup_{j=1}^{i-1}x_j\right)(x_i) = 1\right\}.
\]
Hence, $\glang_0$ is exactly given by the permutations of acyclic sets of edges of size at most 6 having empty intersection with $Y$, and acyclic sets of edges of size exactly 6.
Finally, define another language $\glang_A$ consisting of words of length at most 6 obtained by concatenating words in $\glang_0$ containing at least two letters from $A$ in their support with words in the free monoid $\univ^*$ in a way ensuring that the resulting word contains $A$ in its support if it is not feasible in $\glang_0$,
\[
    \glang_A \triangleq \left\{\alpha x\beta \in\univ^*\mid \alpha x \beta\text{ is simple},|\alpha x\beta|\leq 6, \alpha \in \glang_0,\text{and }A\subseteq(\widetilde{\alpha} + x)\right\}.
\]
Importantly, this means that for any word $\alpha \in \glang_A$ such that $|\alpha| < 6$, $\alpha x \in \glang_A$ for all $x \notin\widetilde{\alpha}$.
Finally, let $\glang \triangleq \glang_0 \cup \glang_A$.
We will show that $\glang$ is an optimistic local poset greedoid for which there is no polymatroid representation.
For what follows, let $r_0$ be the rank function of $\glang_0$ (which we will see is a greedoid as well).

\begin{claim}
    $\glang$ is a greedoid.
\end{claim}
\begin{proof}
    Note $\epsilon \in \glang_0$, and $\glang_0\cup \glang_A$ is clearly hereditary.
    So, we need only examine the exchange property.
    Let $\alpha,\beta \in \glang$ be such that $|\beta| > |\alpha|$.
    Observe, the function restriction $f'|_\univ$ is a monotone, normalized, and submodular function, and so $\glang_0$ is a polymatroid greedoid.
    Hence, exchange holds whenever $\alpha$ and $\beta$ are both contained in $\glang_0$.
    Likewise, by definition of $\glang_A$ there of course exists a letter from $\widetilde{\beta}$ which can be used to augment $\alpha$ whenever $(\alpha,\beta) \in \glang_A\times \glang_A$ or $(\alpha,\beta) \in \glang_A\times\glang_0$.
    So, let $\alpha \in \glang_0\setminus\glang_A$ and $\beta \in \glang_A$.
    Suppose $A\subseteq\sigma_{r_0}\left(\widetilde{\alpha}\right)$.
    Then $|A\cap\widetilde{\alpha}| \geq 2$, meaning that $|A\cap\widetilde{\alpha}| = 2$ by $\alpha \notin \glang_A$.
    Therefore, there is a letter $x \in A\cap(\widetilde{\beta}\setminus\widetilde{\alpha})$ which gives the word $\alpha x\epsilon\in\glang_A$.
    So $\alpha x \in \glang$, and exchange holds in this case.
    Now examine the case where $A\not\subseteq\sigma_{r_0}\left(\widetilde{\alpha}\right)$.
    Recall that the rank closure is the complement of the feasible continuations, hence making $A\setminus\widetilde{\alpha}$ a subset of the continuations of $\alpha$ giving longer words in $\glang_0$.
    Because $\beta \in \glang_A$ implies $A\subseteq \widetilde{\beta}$ it follows then that there is a letter in $\widetilde{\beta}$ which can be appended onto $\alpha$ to obtain a longer word in $\glang_0$.
\end{proof}

\begin{claim}
    $\glang$ possesses the local poset property.
\end{claim}
\begin{proof}
    Let $\alpha,\beta,\gamma$ satisfy the conditions of Definition~\ref{def:local-poset}.
    Importantly, $\widetilde{\alpha},\widetilde{\beta}\subseteq\widetilde{\gamma}$ and $\gamma \in \glang$ implies $|\widetilde{\alpha}\cup\widetilde{\beta}| \leq 6$.
    Then, under these conditions, we see that the intersection $\widetilde{\alpha}\cap\widetilde{\beta}$ and union $\widetilde{\alpha}\cup\widetilde{\beta}$ can always be permuted into another word in $\glang_A$ when $(\alpha,\beta) \in \glang_A\times\glang_A$ due to the construction of $\glang_A$.
    Likewise, the intersection and union can be permuted into a feasible word of $\glang_0$ whenever $(\alpha,\beta)\in\glang_0\times\glang_0$ since $\glang_0$ is a polymatroid greedoid.
    Finally, suppose $\alpha\in\glang_0$ while $\beta \in \glang_A$.
    Notice then that the letters from $\widetilde{\alpha}\setminus\widetilde{\beta}$ can be appended onto $\beta$ in any order to obtain a word in $\glang_A$ supported by $\widetilde{\alpha}\cup\widetilde{\beta}$.
    Likewise, one can observe $(\widetilde{\alpha}\cap\widetilde{\beta})!\cap\glang = \varnothing$ only if $Y \cap \widetilde{\alpha}\cap\widetilde{\beta} \neq \varnothing$.
    However, by the construction of $\univ$ and $\glang_0$ the letters in $Y$ are only feasible continuations (with respect to the greedoid $\glang_0$) of words $\alpha' \in \glang_0$ such that $\widetilde{\alpha}'$ either contains one of $A$, $B$, or $C$, or $|\alpha'| = 5$.
    If $|\alpha| = 6$, then it follows from $\gamma \in \glang$ and $\widetilde{\alpha}\subseteq\widetilde{\gamma}$ that $\widetilde{\alpha} = \widetilde{\gamma}$.
    In such a case, $\widetilde{\alpha}\cap\widetilde{\beta} = \widetilde{\beta}$, which clearly can be permuted into a word in $\glang$.
    So assume $|\alpha| < 6$, meaning that $\widetilde{\alpha}$ contains one of $A$, $B$, or $C$.
    Because $\widetilde{\beta}$ contains $A$, and $\gamma$ is a word of length at most 6 containing $\widetilde{\alpha}\cup\widetilde{\beta}$ in its support, we see $A\subseteq \widetilde{\alpha}$ by pigeonhole principle.
    Hence $\alpha \in \glang_A$, and we've reduced this case to one which we've already solved.
\end{proof}

\begin{claim}
    $\glang$ possesses optimism.
\end{claim}
\begin{proof}
    Select a basic word $\alpha \in \glang$ and letter $x \in\univ$, and by way of contradiction assume $x \notin\Gamma[\alpha]$ for any $\alpha$-prefix.
    The letters of $A\cup B\cup C$ are continuations of $\epsilon$, hence making $x \in Y$.
    Notice, if $\widetilde{\alpha}\cap Y \neq \varnothing$ then there is a letter in $\alpha$ which can be replaced by $x$ without changing the relative ordering of any other letters (thereby making $x$ a continuation of the corresponding prefix).
    So to maintain our assumptions, we see $\widetilde{\alpha}\cap Y =\varnothing$.
    Then let $\alpha'$ be the length 5 prefix of $\alpha$.
    For $\alpha'x \notin \glang$ it must be that $\alpha' \in \glang_0\setminus\glang_A$, hence we see $f'(\widetilde{\alpha}') = 5$.
    Notice that $Y \cap \widetilde{\alpha}' = \varnothing$ as well, so the number of edges in any spanning tree of the subgraph induced by $\widetilde{\alpha}'\cup Y$ is at least one greater than that induced by $\widetilde{\alpha}'$.
    This makes,
    \[
        f'(\widetilde{\alpha}'+ x) - f'(\widetilde{\alpha}') \geq 1,\quad\forall x \in Y.
    \]
    But, $f'(\widetilde{\alpha}'+ x) - f'(\widetilde{\alpha}') \leq 6 - 5 = 1$ by definition of $f'$, and so $\left(f'/\widetilde{\alpha}'\right)(x) = 1$ for all $x \in Y$.
    As a consequence $\alpha' x \in \glang_0$ for all $x \in Y$, a contradiction.
\end{proof}

\begin{proposition}
    There exists an optimistic local poset greedoid for which there is no aligned polymatroid representation.
\end{proposition}
\begin{proof}
    By the preceding claims, $\glang$ is an optimistic local poset greedoid.
    Towards a contradiction, let $\rho:2^\univ\to\mathbb{R}$ be any (possibly non-integral) aligned polymatroid representation of $\glang$.
    Observe $b_1b_2b_3y_1, b_1b_2b_3y_2, b_1b_2b_3y_3 \in \glang$ implies,
    $$\rho(B+ y_1) = \rho(B+y_2) = \rho(B+y_3) =4.$$
    Notice that the set of continuations for $b_1b_2b_3y_1$, $b_1b_2b_3y_2$, and $b_1b_2b_3y_3$ are the same, hence making $b_1b_2b_3y_1 \sim b_1b_2b_3y_2 \sim b_1b_2b_3y_3$.
    This implies $\rho(B\cup Y) = 4$.
    The same is true of $c_1c_2c_3y_1$, $c_1c_2c_3y_2$, and $c_1c_2c_3y_3$, hence $\rho(C\cup Y) = 4$ as well.
    It can be seen that $\rho(B\cup C\cup Y) = 6$, since $b_1b_2b_3c_1c_2c_3 \in \glang$ and $y_1,y_2,y_3 \in \kappa[b_1b_2b_3c_1c_2c_3]$ as $b_1b_2b_3c_1c_2c_3$ is basic and none of $y_1$, $y_2$, or $y_3$ are loops.
    Then because submodularity makes,
    \begin{equation*}
    \overbrace{\rho(B\cup Y) + \rho(C\cup Y)}^{4+4}\geq \underbrace{\rho(B\cup C\cup Y)}_{6} + \rho(Y),
    \end{equation*}
    it follows that $\rho(Y) \leq 2$.
    However, because $a_1a_2a_3y_1y_2y_3 \in \glang$, we see $(\rho/A)(Y) = 3$.
    In particular, this means that $(\rho/A)(Y) > (\rho/\varnothing)(Y) = 2$, and so $\rho$ does not satisfy the law of diminishing returns (a contradiction).
\end{proof}

\begin{remark}\label{rem:change}
    In \cite{korte1988intersection}, Korte and Lovász limit the length of the feasible words using a truncation operation.
    This would make basic words such that no prefix possesses all 3 letters from one of $A, B, C$, like $a_1b_1c_1a_2b_2c_2$ for example, such that there is no prefix possessing a letter from $Y$ as a feasible continuation.
    In contrast, we limit the length by operating on the function $f$ measuring the size of spanning trees in the induced subgraphs (effectively by intersecting the corresponding graphic matroid with a \emph{uniform matroid} of rank 6).
    Specifically, the function $f'(X) = \min\{f(X), 6\}$, which is a polymatroid representation of $\glang_0$, is such that the marginal return of a letter $x \in Y$ under such words of length 5 is always equal to one, hence making $x$ a continuation of such words in $\glang_0$ for preserving optimism.
\end{remark}

\section{Proof of Proposition~\ref{prop:representation->local-poset}}\label{app:deferred-lp}

We claimed that the argument given in \cite{korte1985polymatroid} for Proposition~\ref{prop:representation->local-poset} doesn't require integrality.
So, now we have to make good on that promise.
We follow parts of the original argument from \cite{korte1985polymatroid}.
This requires some definitions which are equivalent to the interval property and local poset property which are not used in any of our original proofs (see \cite{korte2012greedoids} for reference).

The interval property is equivalently described by following:
For all $X\subseteq Y \subseteq Z$ such that all of $X!\cap\glang$, $Y!\cap\glang$, and $Z!\cap \glang$ are nonempty and $x \notin Z$,
\begin{equation}\label{eq:alt-int}
        (X + x)!\cap \glang \neq \varnothing\text{ and }(Z + x)! \cap \glang\neq \varnothing \implies (Y+x)!\cap \glang \neq \varnothing.
\end{equation}
One can use this to prove the following via the law of diminishing returns.
\begin{proposition}[\cite{korte1985polymatroid}]\label{prop:int}
    Suppose $\glang$ is a greedoid with any polymatroid representation $\rho:2^\univ\to\mathbb{R}$, then it possesses the interval property. 
\end{proposition}
\begin{proof}
    The argument from \cite{korte1985polymatroid} uses only the fact that the marginal return of any continuation under a feasible word is always one.
    Specifically, let $X \subseteq Y \subseteq Z$ be such that $X!\cap\glang$, $Y!\cap\glang$, and $Z!\cap\glang$ are non-empty.
    Then, take any $x\notin Z$ satisfying the premise of Eq.~\ref{eq:alt-int}.
    From the law of diminishing returns,
    \[
    1 = (\rho/X)(x) \geq (\rho/Y)(x) \geq (\rho/Z)(x) = 1.
    \]
    Hence, for all $\alpha \in Y!\cap\glang$
    it follows that $(\rho/\widetilde{\alpha})(x) = 1$.
    Thus $x \in \Gamma[\alpha]$ by definition of representation, so the intersection on the right hand side of Eq.~\ref{eq:alt-int} non-empty since $\alpha x \in (Y +x)!\cap\glang$.
\end{proof}

We also need the following.

\begin{lemma}\label{lem:std}
    Suppose $\glang$ is a greedoid with any polymatroid representation $\rho:2^\univ\to\mathbb{R}$.
    Fix $X,Y\subseteq\univ$, and suppose $X\subseteq Y$.
    Then, if $Y!\cap\glang\neq\varnothing$, it follows that $\rho(X) \geq |X|$, with equality if and only if $X!\cap\glang\neq\varnothing$.
\end{lemma}
\begin{proof}
    Induct over $|X|$.
    Clearly this is true when $X=\varnothing$.
    Now, let $\alpha \in Y^*\cap\glang$ be a maximal length word satisfying $X\not\subseteq \widetilde{\alpha}$.
    Then, by the exchange property there exists $z \in Y\setminus \widetilde{\alpha}$ such that $\alpha z \in\glang$.
    Observe, because $\alpha$ is a maximal word satisfying $X\not\subseteq\widetilde{\alpha}$, we have $X\subseteq\widetilde{\alpha z}$.
    This makes all of $z \in X$,$X - z \subseteq\widetilde{\alpha}$, and
    $X\cup\widetilde{\alpha} = \widetilde{\alpha z}$ true.
    Because $\rho$ is a representation, we also have all of $\rho(X\cup\widetilde{\alpha}) = |\alpha| + 1$, $\rho(X\cup\widetilde{\alpha}) = |X\cup\widetilde{\alpha}|$, and $\rho(\widetilde{\alpha}) = |\alpha|$ in particular.

    Our construction also makes $X\cap\widetilde{\alpha} \subset X$, and so $\rho(X\cap\widetilde{\alpha}) \geq |X\cap\widetilde{\alpha}|$ by inductive hypothesis.
    Then with submodularity,
    \begin{equation}\label{eq:uiuitsi}
        \rho(X) \geq \rho(X\cap\widetilde{\alpha}) + \rho(X\cup\widetilde{\alpha}) - \rho(\widetilde{\alpha}) \geq |X\cap\widetilde{\alpha}| + |X\cup\widetilde{\alpha}| - |\widetilde{\alpha}| = |X|.
    \end{equation}
    This shows the first part of the result.
    Now for the second, first assume $\rho(X) = |X|$.
    Because $(\rho/\widetilde{\alpha})(z) = 1$ and $X-z\subseteq\widetilde{\alpha}$, it follows by the law of diminishing returns that $(\rho/X-z)(z)\geq 1$.
    Combining this with our inductive hypothesis,
    \[
        |X| - 1 \leq \rho(X-z) \leq \rho(X) - 1 = |X| - 1.
    \]
    So $\rho(X-z) = |X| -1$, which makes $(X-z)!\cap\glang\neq\varnothing$ by our hypothesis.
    Therefore, since $(\rho/X-z)(z) = 1$, it follows from the definition of representation that $X!\cap\glang\neq\varnothing$ as well.
    Now suppose $X!\cap\glang\neq\varnothing$.
    Then $\rho(X) = |X|$ by assuming $\rho$ is a representation.
    Hence the lemma.
\end{proof}

Now, a greedoid is a local poset greedoid if and only if it is an interval greedoid with the \emph{local intersection property}.
That is, for all $X\subseteq \univ$ and $y,z\in X$,
\[
    X!\cap\glang\neq\varnothing\text{ and }(X-y)!\cap\glang\neq\varnothing\text{ and }(X-z)!\cap\glang\neq\varnothing \implies (X\setminus\{y,z\})!\cap\glang\neq \varnothing.
\]
See Theorem 2.2 in Ch. 7 of \cite{korte2012greedoids}, for example.
So, in light of Proposition~\ref{prop:int} we need only verify the above.

\begin{proof}[Proof of Proposition~\ref{prop:representation->local-poset}]
    Let $\rho:2^\univ\to\mathbb{R}$ be any (possibly non-integral or non-aligned) representation, and select $X\subseteq \univ$ and $y,z\in\univ$ with $X!\cap\glang\neq\varnothing$, $(X-y)!\cap\glang\neq\varnothing$, and $(X-z)!\cap\glang\neq\varnothing$.
    By Lemma~\ref{lem:std} and submodularity this construction makes,
    \[
        |X| - 2 \leq \rho(X\setminus\{y,z\}) \leq \rho(X-y) + \rho(X-z) - \rho(X) = |X| - 2.
    \]
    But then, $|X| - 2 = \rho(X\setminus\{y,z\})$, so $(X\setminus\{y,z\})!\cap\glang \neq \varnothing$ once again by Lemma~\ref{lem:std}.
\end{proof}

\section{Greediness is Closed Under Contraction Minors}\label{app:greedy-minors}

First, define the \emph{strong exchange property}:
For all $\alpha \in \glang$ and basic words $\beta \in \glang$ with $\widetilde{\alpha}\subseteq\widetilde{\beta}$, if $x \notin\widetilde{\beta}$ is such that $x \in \Gamma[\alpha]$, then there exists $y\in\widetilde{\beta}\setminus\widetilde{\alpha}$ such that $y \in \Gamma[\alpha]$ and $(\widetilde{\beta} - y + x)!\cap\glang \neq \varnothing$.
In \cite{korte1984greedoids,goetschel1986linear}, it was shown that this is necessary and sufficient for correctness of the greedy algorithm on linear optimization over basic words.
We prove the following:
\begin{proposition}
    Let $\glang$ be an interval greedoid and $F\in\glang/\mathord{\sim}$.
    If the greedy algorithm correctly solves any linear optimization over the basic words of $\glang$, then it correctly solves any linear over the basic words of $\glang/F$ as well.
\end{proposition}
\begin{proof}
    Firstly, examine $\glang/F$, and select any $\alpha\in\glang/F$, basic $\beta\in\glang/F$, and $x\notin\widetilde{\beta}$ which is a continuation of $\alpha$ in $\glang/F$, as in the statement of the strong exchange property.
    Well, for any $\gamma\in F$ we have $\gamma\alpha$ and $\gamma\beta$ are feasible.
    Furthermore, we also have $x\in\Gamma[\gamma \alpha]$.
    This means $x$ is not in the support of $\gamma$, so $x \notin \widetilde{\gamma\beta}$.
    Then, there exists $y\in\widetilde{\gamma\beta}\setminus\widetilde{\gamma\alpha}$ such that $\gamma \alpha y\in\glang$ and $\widetilde{\gamma\beta} - y + x$ can be permuted into a $\glang$-feasible word.
    This makes $y$ a continuation of $\alpha$ in $\glang/F$.
    But also, our construction ensures $y \notin\gamma$.
    So, by the exchange property there is some permutation $\pi \in (\widetilde{\beta} - y + x)!$ such that $\gamma\pi \in \glang$.
    Hence $\pi$ is basic in $\glang/F$, and the strong exchange property holds.
\end{proof}

\end{document}

%% file: macros.tex
\newcommand{\univ}{\Sigma}
\newcommand{\glang}{\Lambda}

\newcommand{\irr}{\mathop{\textsf{\textup{irr}}}}
\newcommand{\overbar}[1]{\mkern 1.5mu\overline{\mkern-1.5mu#1\mkern-1.5mu}\mkern 1.5mu}
\newcommand{\greedyrep}{\ddot{\rho}}

\newcommand{\todo}[1]{\textcolor{red}{TODO: #1}}

\makeatletter
\providecommand{\bigsqcap}{%
	\mathop{%
		\mathpalette\@updown\bigsqcup
	}%
}
\newcommand*{\@updown}[2]{%
	\rotatebox[origin=c]{180}{$\m@th#1#2$}%
}
\makeatother

\makeatletter
\newcommand{\vast}{\bBigg@{4}}
\newcommand{\Vast}{\bBigg@{5}}
\makeatother

\DeclareMathOperator{\raisedrightarrow}{\raisebox{+0.15ex}{$\rightarrow$}}
\makeatletter
\newcommand{\rightarroweq}{\mathrel{\mathpalette\rightarroweq@\relax}}
\newcommand{\rightarroweq@}[2]{%
  \vtop{
      \sbox\z@{$\m@th#1\raisedrightarrow$}%
    \ialign{%
      \hfil##\hfil\cr
      \copy\z@\cr
      \noalign{\nointerlineskip\kern0.7\ht\z@}
      \makebox[\wd\z@][s]{$\m@th#1\relbar\hss\relbar$}\cr
      \noalign{\kern-0.7\ht\z@}
    }%
  }%
}
\makeatother

\DeclareMathOperator{\sraisedrightarrow}{\raisebox{+0.075ex}{$\scriptstyle\rightarrow$}}
\makeatletter
\newcommand{\srightarroweq}{\mathrel{\mathpalette\srightarroweq@\relax}}
\newcommand{\srightarroweq@}[2]{%
  \vtop{
      \sbox\z@{$\m@th#1\sraisedrightarrow$}%
    \ialign{%
      \hfil##\hfil\cr
      \copy\z@\cr
      \noalign{\nointerlineskip\kern0.7\ht\z@}
      \makebox[\wd\z@][s]{$\m@th#1\relbar\hss\relbar$}\cr
      \noalign{\kern-0.7\ht\z@}
    }%
  }%
}
\makeatother

\DeclareMathOperator{\sraisedleftarrow}{\raisebox{+0.075ex}{$\scriptstyle\leftarrow$}}
\makeatletter
\newcommand{\sleftarroweq}{\mathrel{\mathpalette\sleftarroweq@\relax}}
\newcommand{\sleftarroweq@}[2]{%
  \vtop{
      \sbox\z@{$\m@th#1\sraisedleftarrow$}%
    \ialign{%
      \hfil##\hfil\cr
      \copy\z@\cr
      \noalign{\nointerlineskip\kern0.7\ht\z@}
      \makebox[\wd\z@][s]{$\m@th#1\relbar\hss\relbar$}\cr
      \noalign{\kern-0.7\ht\z@}
    }%
  }%
}
\makeatother

\newcommand{\ideal}[1]{\mathop{\downarrow}#1}

\DeclareMathOperator*{\argmin}{arg\,min}

\newcommand{\rank}{r}

\providecommand{\L}{}
\renewcommand{\L}{\mathscr{L}}

\newcommand{\proposal}[1]{\ignorespaces}

%% file: tikz/optimistic.tikz
\begin{tikzpicture}[set/.style={very thick}]
\tikzstyle{break} = [align=center]
    \fill[gray!20] (2.5,2.5) -- (2.5,8.5) -- (9,8.5) -- (9,2.5);
    \draw (2.5,8.5) -- (9,8.5);
    
    \fill[gray!40] (2.5,7.5) rectangle (8,2.5);
    \fill[gray!40] (8,7) rectangle (8.5,2.5);
    \fill[gray!40,rounded corners=10pt] (2.5,7.5) rectangle (8.5,2.5);
    \draw[rounded corners=10pt] (2.5,7.5) -- (8.5,7.5) -- (8.5, 2.5);

    \draw[rounded corners=10pt] (-2.5, -.5) rectangle (14,10.5);

    \draw[rounded corners=10pt] (-2,0) rectangle (9.5,10);

    \draw[rounded corners=10pt] (2.5,.5) rectangle (9,9);

    \draw[rounded corners=10pt] (3,1.5) rectangle (8,5.5);

    \draw[rounded corners=10pt] (-1.5, 2.5) rectangle (13.5,9.5);

    \draw[rounded corners=10pt] (-1, 3) rectangle (13, 6.5);

    \draw[rounded corners=10pt] (-.5,3.5) rectangle (5,6);

    \draw[rounded corners=10pt] (3.5, 4) rectangle (7.5, 5);

    \draw (5.75, 1) node[break] {Local Poset Greedoid};
    \draw (11.5, 8) node[break] {\textbf{Optimistic}\\\textbf{Greedoid}};
    \draw (0.25, 1.25) node[break] {Interval Greedoid};
    \draw (5.5,2) node[break] {Distributive Supermatroid};
    \draw (5.75,8) node[break] {\textbf{Aligned Polymatroid Greedoid}};
    \draw (11.25,4.75) node[break] {Strong-Exchange\\Greedoid};
    \draw (1,4.75) node[break] {Antimatroid};
    \draw (6.25, 4.5) node[break] {Matroid};
    \draw (11.75, 1.125) node[break] {Greedoid};
    \draw (5.5,7) node[break] {\textbf{Strong Polymatroid Greedoid}};
\end{tikzpicture}

%% file: tikz/udb-graph.tikz
\begin{tikzpicture}
    \node[draw, circle] (v1) {};
    \node[draw, circle] (v2) [above left=of v1] {};
    \node[draw, circle] (v3) [below left=of v1] {};
    \node (inv1) [below=of v1, yshift=-35pt] {};

    \node[draw, circle] (v4) [right=of v1] {$s$};

    \draw (v1) -- node[midway, above right] {$d$} (v2);
    \draw (v2) to[out=225,in=135] node[midway, left] {$b$} (v3) ;
    \draw (v3) -- node[midway, below right] {$c$} (v1);

    \draw (v1) -- node[midway, above] {$a$} (v4);
\end{tikzpicture}

%% file: tikz/udb-lattice-of-flats.tikz
\begin{tikzpicture}
    \tikzstyle{S}=[rectangle, draw=black, rounded corners=5pt]
    \node[S] (empty) {$\{\epsilon\}$};
    \node[S] (a) [above=of empty, yshift=-5pt] {$\{a\}$};
    \node[S] (b) [above left=of a, yshift=-5pt] {$\{ac\}$};
    \node[S] (c) [above right=of a, yshift=-5pt] {$\{ad\}$};
    \node (inv) [above=of a] {};
    \node[S] (top) [above=of inv] {$\{acb, acd, adb, adc\}$};

    \foreach \from/\to in {empty/a, a/b, a/c, b/top, c/top}
    \draw (\from) -- (\to);
\end{tikzpicture}

%% file: tikz/pm-spans.tikz
\begin{tikzpicture}
    \tikzstyle{S}=[rectangle, draw=black, rounded corners=5pt]
    \node[S] (empty) {$\varnothing$};
    \node[S] (ap) [above=of empty, yshift=-5pt] {$\{a\}$};
    \node[S] (adp) [above=of ap, yshift=-5pt] {$\{a,d\}$};
    \node[S] (acp) [left=of adp] {$\{a,c\}$};
    \node[S] (topp) [above=of adp, yshift=-5pt] {$\{a,b,c,d\}$};
    \node[S] (bp) [right=of adp] {$\{b\}$};
 
    \foreach \from/\to in {empty/ap, empty/bp, ap/adp, ap/acp, acp/topp, adp/topp, bp/topp}
    \draw (\from) -- (\to);
\end{tikzpicture}

%% file: tikz/forking.tex
\newcommand{\drawSolidDashedEdge}[2]{
    \coordinate (start) at (#1);
    \coordinate (end) at (#2);

    \path (start) -- (end) coordinate[pos=0.2] (firstsolid)
                        coordinate[pos=0.8] (secondsolid);

    \draw[solid] (#1) -- (firstsolid);
    
    \draw[thick,loosely dotted] (firstsolid) -- (secondsolid);
    
    \draw[solid] (secondsolid) -- (#2);
}
\newcommand{\drawSolidDashedEdgeTwo}[2]{
    \coordinate (start) at (#1);
    \coordinate (end) at (#2);

    \path (start) -- (end) coordinate[pos=0.3] (firstsolid)
                        coordinate[pos=0.7] (secondsolid);

    \draw[solid] (#1) -- (firstsolid);
    
    \draw[thick,loosely dotted] (firstsolid) -- (secondsolid);
    
    \draw[solid] (secondsolid) -- (#2);
}

\begin{tikzpicture}
    \tikzstyle{S}=[rectangle, draw=black, rounded corners=5pt]

    \node[S] (Fp) at (1.75,.5) {$F'$};
    \node[S] (F) at (-1.75, -.25) {$F$};

    \node[S] (m) at (0, -2.5) {$F\sqcap F' =[\mu]$};
    \node[S] (mx) at (.75, -1.25) {$[\mu x]$};
    \node[S] (Fx) at (-1, 1) {$F \sqcup [\mu x]$};
    \node[S] (join) at (0, 2.75) {$F \sqcup F'$};

    \draw (m) -- (mx) node [midway, right] {$x$};
    \draw (F) -- (Fx) node [midway, left] {$x$};

    \drawSolidDashedEdgeTwo{mx}{Fp};
    \drawSolidDashedEdge{m}{F};
    \drawSolidDashedEdge{Fp}{join};
    \drawSolidDashedEdgeTwo{Fx}{join};
\end{tikzpicture}

%% file: tikz/gr-spans.tikz
\begin{tikzpicture}
    \tikzstyle{S}=[rectangle, draw=black, rounded corners=5pt]
    \node[S] (empty) {$\varnothing$};
    \node (inv) [below=of empty, yshift=25pt] {};
    \node[S] (bp) [above left=of empty, xshift=20pt, yshift=-5pt] {$\{b\}$};
    \node[S] (ap) [left=of bp] {$\{a\}$};
    \node[S] (cp) [above right=of empty, xshift=-20pt, yshift=-5pt] {$\{c\}$};
    \node[S] (dp) [right=of cp] {$\{d\}$};

    \node[S] (abp) [above=of ap, yshift=-5pt] {$\{a,b\}$};
    \node[S] (adp) [above=of cp, yshift=-5pt] {$\{a,d\}$};
    \node[S] (acp) [above=of bp, yshift=-5pt] {$\{a,c\}$};
    \node[S] (bcdp) [above=of dp,yshift=-5pt] {$\{b,c,d\}$};
    \node[S] (topp) [above=of empty, yshift=82.5pt] {$\{a,b,c,d\}$};
 
    \foreach \from/\to in {empty/ap, empty/bp,empty/cp,empty/dp,ap/abp,ap/acp, ap/adp, ap/acp, bp/abp,bp/bcdp,cp/acp,cp/bcdp,dp/adp,dp/bcdp, acp/topp, adp/topp, bcdp/topp, abp/topp}
    \draw (\from) -- (\to);
\end{tikzpicture}

%% file: tikz/galois.tikz
\begin{tikzpicture}
    \tikzstyle{S}=[rectangle, draw=black, rounded corners=5pt]
    \node[S] (e) {$\{\epsilon\}$};
    \node[S] (a) [above=of e, yshift=-5pt] {$\{a\}$};
    \node[S] (ac) [above left=of a, yshift=-5pt] {$\{ac\}$};
    \node[S] (ad) [above right=of a, yshift=-5pt] {$\{ad\}$};
    \node (inv) [above=of a] {};
    \node[S] (topflat) [above=of inv] {$\{acb, acd, adb, adc\}$};

    \foreach \from/\to in {e/a, a/ac, a/ad, ac/top, ad/top}
    \draw (\from) -- (\to);

    \node[S] (empty) [right=of e, xshift=150pt, yshift=25pt] {$\varnothing$};
    \node[S] (ap) [above=of empty, yshift=-5pt] {$\{a\}$};
    \node[S] (adp) [above=of ap, yshift=-5pt] {$\{a,d\}$};
    \node[S] (acp) [left=of adp] {$\{a,c\}$};
    \node[S] (topp) [above=of adp, yshift=-5pt] {$\{a,b,c,d\}$};
    \node[S] (bp) [right=of adp] {$\{b\}$};

    \node (lrho) [below=of empty,yshift=22.5pt] {\large $\L(\rho)$};
    \node (lg) [above=of topflat,yshift=-22.5pt] {\large $\L(\glang)$};
    
    \foreach \from/\to in {empty/ap, empty/bp, ap/adp, ap/acp, acp/topp, adp/topp, bp/topp}
    \draw (\from) -- (\to);

    \foreach \from/\to in {e/empty, a/ap, ac/acp, ad/adp, topflat/topp}
    \draw[latex-latex, thick, dotted, gray] (\from) -- (\to);

    \draw[-latex, thick, dotted, gray] (bp) -- (e);

    \node (lal) [right=of lg, xshift=12.5pt] {};
    \node (lar) [right=of lal, xshift=37.5pt] {};

    \draw[-latex] (lal) -- (lar)  node [midway,above] {$\varphi_* = \sigma_\rho \circ \kappa$};

    \node (uar) [left=of lrho, xshift=-15pt] {};
    \node (ual) [left=of uar, xshift=-40pt] {};

    \draw[-latex] (uar) -- (ual) node [midway, below] {$\varphi^* = \kappa^{-1}$};
\end{tikzpicture}

%% file: tikz/cycle-graph.tikz
\newcommand{\drawSolidDashedEdge}[2]{
    \coordinate (start) at (#1);
    \coordinate (end) at (#2);

    \path (start) -- (end) coordinate[pos=0.2] (firstsolid)
                        coordinate[pos=0.8] (secondsolid);

    \draw[solid] (#1) -- (firstsolid);
    
    \draw[thick,loosely dotted] (firstsolid) -- (secondsolid);
    
    \draw[solid] (secondsolid) -- (#2);
}

\begin{tikzpicture}[rotate=77]  
    \def \n {7} 
    \def \radius {3cm}
    \def \angle {360/\n}

    \foreach \i in {0,1,...,6} {
        \coordinate (v\i) at ({\i * \angle + 90}:\radius); 
        
        \ifnum \i=2
            \node at (v\i) [draw,circle] (n\i) {$s$};
        \else
            \node at (v\i) [draw,circle] (n\i) {};
        \fi
    }

    \foreach \i in {0,1,...,5} {
        \pgfmathtruncatemacro{\nexti}{mod(\i+1,\n)}

        \ifnum \i=0
            \drawSolidDashedEdge{n\i}{n\nexti};
        \else\ifnum \i=3
            \drawSolidDashedEdge{n\i}{n\nexti};
        \else
            \draw (n\i) -- (n\nexti);
        \fi\fi
    }
    \draw (n6) -- (n0);
    
    \node[above] at ($(v6)!0.5!(v5)$) {$x$};  

    \path (v2) -- (v3) coordinate[pos=0.5] (a1);
    \path (v2) -- (v1) coordinate[pos=0.5] (b1);
    \node[below right] at (a1) {$a_1$};
    \node[below left] at (b1) {$b_1$};

    \path (v4) -- (v5) coordinate[pos=0.5] (ai);
    \path (v0) -- (v6) coordinate[pos=0.5] (bj);
    \node[above right] at (ai) {$a_i$};
    \node[above left] at (bj) {$b_j$};

\end{tikzpicture}